\documentclass[12pt, reqno]{amsart}
\usepackage{amssymb,dsfont}
\usepackage{eucal}
\usepackage{amsmath}
\usepackage{amscd}
\usepackage[dvips]{color}
\usepackage{multicol}
\usepackage[all]{xy}           
\usepackage{graphicx}
\usepackage{color}
\usepackage{colordvi}
\usepackage{xspace}
\usepackage{txfonts}
\usepackage{lscape}
\usepackage{tikz} 
\numberwithin{equation}{section}
\usepackage[shortlabels]{enumitem}
\usepackage{ifpdf}
\ifpdf
\usepackage[colorlinks,final,backref=page,hyperindex]{hyperref}
\else
\usepackage[colorlinks,final,backref=page,hyperindex]{hyperref}
\fi
\usepackage{tikz}
\usepackage[active]{srcltx}

\makeatletter

\numberwithin{equation}{section}
\newtheorem{theo}{Theorem}[section]
\newtheorem{defi}[theo]{Definition}
\newtheorem{coro}[theo]{Corollary}
\newtheorem{lemm}[theo]{Lemma}

\newtheorem{rem}[theo]{Remark}
\newtheorem{prop}[theo]{Proposition}

\newtheorem{rema}[theo]{Remark}

\newcommand{\End}{{\mathrm{End}}}

\newcommand{\cc}{\mathbb{C}}

\newcommand{\z}{\mathbb{Z}}

\newcommand{\mi}{\mathbf{i}}
\newcommand{\mk}{\mathbf{k}}

\newcommand{\mj}{\mathbf{j}}

\newcommand{\bmi}{\underline{\mathbf{i}}}
\newcommand{\bmj}{\underline{\mathbf{j}}}
\newcommand{\bmk}{\underline{\mathbf{k}}}

\newcommand{\Ind}{\mbox{Ind}}

\def\NO{\mbox{\,$\circ\atop\circ$}\,}

\begin{document}
	
	\title{Smooth modules  over the N=1 Bondi-Metzner-Sachs  superalgebra}

	\author{Dong Liu}
	\address{Department of Mathematics, Huzhou University, Zhejiang Huzhou, 313000, China}
	\email{liudong@zjhu.edu.cn}
	
	\author{Yufeng Pei}
	\address{Department of Mathematics, Huzhou University, Zhejiang Huzhou, 313000, China} \email{pei@shnu.edu.cn}
	
	\author{Limeng Xia}
	
	\address{Institute of Applied System Analysis, Jiangsu University, Jiangsu Zhenjiang, 212013,
		China}\email{xialimeng@ujs.edu.cn}

\author[Zhao]{Kaiming Zhao}
\address{Department of Mathematics, Wilfrid Laurier University, Waterloo, ON, Canada N2L3C5}\email{kzhao@wlu.ca}

 \date{\today}



	\maketitle

	\begin{abstract}
In this paper, we present a determinant formula for the contravariant form on Verma modules over the N=1 Bondi-Metzner-Sachs (BMS)  superalgebra. This formula establishes  a necessary and sufficient condition for the irreducibility of the Verma modules.
We then introduce and characterize a class of simple smooth modules that generalize both Verma and Whittaker modules over the N=1 BMS superalgebra.
We also utilize the Heisenberg-Clifford vertex superalgebra to construct a free field realization for the N=1 BMS superalgebra. This free field realization allows us to obtain a family of natural smooth modules over the N=1 BMS superalgebra, which includes Fock modules and certain Whittaker modules.

 \end{abstract}	
	
	\tableofcontents
	
	\setcounter{section}{0}
	
	\section{Introduction}
	
	\allowdisplaybreaks
	
The study of conformal field theory relies heavily on the Virasoro algebra, which is the central extension of the Witt algebra. The representation theory of this algebra is an active research area in mathematical  and theoretical physics. In particular, simple Virasoro modules have been extensively explored in both mathematics and physics literature (see \cite{IK3}, \cite{MZ1}, \cite{MZ2}, \cite{MZ3} and references therein). Additionally, superconformal extensions of the Virasoro algebra known as super Virasoro algebras play an important role in superstring theory and superconformal field theory. Significant research has also been conducted on simple super Virasoro modules \cite{IK1}, \cite{IK2}, \cite{LPX1}, \cite{LPX2}, \cite{MR}.

The Bondi-Metzner-Sachs (BMS) algebra was proposed to characterize the infinite-dimensional symmetries of asymptotically flat spacetimes at null infinity and was initially discovered in \cite{BBM}, \cite{Sa1}, \cite{Sa2}. The BMS algebra can be also obtained by taking a particular contraction of two copies of the Virasoro algebra.  Notably, the BMS algebra appeared  in the Galilean conformal field \cite{BGMM}, the classification of simple vertex operator algebras of moonshine type \cite{ZD} and  two-dimensional statistical systems \cite{HSSU}. Representations of the BMS algebra have been extensively discussed in references such as \cite{AR}, \cite{BSZ}, \cite{BO1}, \cite{BO2}, \cite{JP}, \cite{JPZ}, \cite{JZ}, \cite{Ob}, \cite{Ra}, \cite{ZD}.

The {\rm N=1} BMS superalgebra $\mathfrak{g}$, a minimal supersymmetric extension of the BMS algebra, plays a key role in describing asymptotic supergravity in three-dimensional flat spacetime \cite{BDMT}. Simple $\mathfrak{g}$-modules with finite-dimensional weight spaces were classified in \cite{DCL}, and Whittaker $\mathfrak{g}$-modules were introduced and studied in \cite{DGL1}. A class of simple non-weight $\mathfrak{g}$-modules was constructed and classified in \cite{CDYS}. Recent progress includes a free field realization of $\mathfrak{g}$ using the $\beta$-$\gamma$ and $b$-$c$ ghost systems \cite{BJMN}.

The study of smooth modules is central to representation theory for $\mathbb{Z}$ or $\frac{1}{2}\mathbb{Z}$-graded Lie superalgebras, due to connections with vertex operator superalgebras \cite{Li}, \cite{Li1}. However, completely classifying smooth modules remains an open challenge. Simple smooth modules have been classified for the Virasoro algebra and the N=1 Neveu-Schwarz algebra \cite{MZ2}, \cite{LPX1}. Partial classification results also exist for other Lie (super)algebras, including the N=1 Ramond algebra \cite{C}, the twisted Heisenberg-Virasoro algebra \cite{CG}, \cite{Gao}, \cite{TYZ}, the mirror Heisenberg-Virasoro algebra \cite{LPXZ0}, \cite{TYZ}, the Ovsienko-Roger algebra \cite{LPX3}, the planar Galilean conformal algebra \cite{GG}, \cite{CY}, the Fermion-Virasoro algebra \cite{XZ}, and the superconformal current algebra \cite{LPXZ}.

Motivated by these advances, we systematically study smooth modules over the {\rm N=1} BMS superalgebra, particularly  including Verma, Whittaker, and Fock  modules. 

First, using an anti-involution of $\mathfrak{g}$, we define a contravariant symmetric bilinear form on Verma modules over $\mathfrak{g}$. We compute the Gram matrix to obtain a determinant formula and determine necessary and sufficient conditions for simplicity of Verma modules (Theorem \ref{main1}). Surprisingly, these computations are much simpler than the Kac determinant formula for the Virasoro algebra \cite{KRR}, despite $\mathfrak{g}$ being more complicated.

Clearly, the Verma modules and the Whittaker modules~\cite{DGL1} are very special cases of smooth modules. It is natural to consider studying general smooth modules over the $N=1$ BMS superalgebra $\mathfrak{g}$. In contrast to Verma and Whittaker modules, the structure of general smooth $\mathfrak{g}$-modules is more complex.
In this paper, we construct (Theorem~\ref{main3}) and classify simple smooth $\mathfrak{g}$-modules under certain conditions (Theorem~\ref{main4}). We establish a one-to-one correspondence between simple smooth $\mathfrak{g}$-modules and simple modules of a family of finite-dimensional solvable Lie superalgebras associated with $\mathfrak{g}$. Unlike the smooth module classification for the Virasoro algebra~\cite{MZ2} and the $N=1$ Neveu-Schwarz algebra~\cite{LPX2}, our classification result requires additional injectivity conditions. Removing these conditions to obtain a complete smooth $\mathfrak{g}$-module classification remains an open problem.

Finally, we show that there is an isomorphism between the categories of smooth $\mathfrak{g}$-modules and {\rm N=1} BMS vertex superalgebra modules. This allows employing powerful vertex algebra tools to study smooth $\mathfrak{g}$-modules. We demonstrate a nontrivial homomorphism from the {\rm N=1} BMS vertex algebra to the Heisenberg-Clifford vertex algebra, providing a free field realization of $\mathfrak{g}$ (Theorem \ref{main2}). Using this, we construct natural smooth $\mathfrak{g}$-modules including Fock modules and certain Whittaker modules. Our results provide new perspectives on representation theory of the {\rm N=1} BMS superalgebra  with potential applications in asymptotic gravity symmetries.

The structure of this paper is as follows: Section 2 reviews some notations and preliminaries.
Section 3 presents a formula for the determinant of the contravariant form on Verma modules (Lemma \ref{l3.1}), which can be used to determine the necessary and sufficient conditions for a Verma module to be simple (Theorem \ref{main1}).
Section 4 constructs simple smooth modules for $\frak g$  (Theorem \ref{main3})and classifies simple smooth modules under certain conditions  (Theorem \ref{main4}).
Section 5  provides a free field realization of $\frak g$  (Theorem \ref{main2}). Using this free field realization, some smooth modules over $\frak  g$ are obtained from smooth modules over the Heisenberg-Clifford superalgebra.

Throughout this  paper, we will use the following notations: $\cc$, ${\mathbb N}$, $\z_+$, and $\z$ refer to the sets of complex numbers, non-negative integers, positive integers, and integers, respectively. We consider a $\z_2$-graded vector space $V = V_{\bar 0} \oplus V_{\bar 1}$, where an element $u\in V_{\bar 0}$ (respectively, $u\in V_{\bar 1}$) is called even (respectively, odd). We define $|u|=0$ if $u$ is even and $|u|=1$ if $u$ is odd. The elements in $V_{\bar 0}$ or $V_{\bar 1}$ are referred to as homogeneous, and whenever $|u|$ is used, it means that $u$ is homogeneous.


\section{Notations and preliminaries}
In this section, we will review the  notations and results associated with  the  N=1 BMS superalgebra.


	\begin{defi}[\cite{BDMT}]
		The   {\bf N=1 BMS superalgebra} 
		$$\frak g=\bigoplus_{n\in\z}\cc L_n\oplus\bigoplus_{n\in\z}\cc M_n\oplus\bigoplus_{n\in\z+\frac{1}{2}}\cc Q_n\oplus\cc {\bf c}_1\oplus\cc {\bf c}_2$$
		is a Lie superalgebra, where 
		$$
		\frak g_{\bar0}=\bigoplus_{n\in\z}\cc L_n\oplus\bigoplus_{n\in\z}\cc M_n\oplus\cc  {\bf c}_1\oplus\cc  {\bf c}_2,\quad \frak g_{\bar1}=\bigoplus_{r\in\z+\frac12} \cc  Q_r,
		$$
		with the following commutation relations:
		\begin{align*}\label{SBMS}
			&{[L_m, L_n]}=(m-n)L_{m+n}+{1\over12}\delta_{m+n, 0}(m^3-m){\bf c}_1,\nonumber\\
			&{[L_m, M_n]}=(m-n)M_{m+n}+{1\over12}\delta_{m+n, 0}(m^3-m){\bf c}_2,\nonumber\\
			&{[Q_r, Q_s]}=2M_{r+s}+{1\over3}\delta_{r+s, 0}\left(r^2-\frac14\right){\bf c}_2,\\
			&{[L_m, Q_r]}=\left(\frac{m}{2}-r\right)Q_{m+r},\nonumber\\
			&{[M_m,M_n]}=[M_n,Q_r]=0,\\
			&[{\bf c}_1,\frak g]=[{\bf c}_2, \frak g]=0, \nonumber
		\end{align*} for any $m, n\in\z, r, s\in\z+\frac12$.
	\end{defi}

Note that the even part $\mathfrak{g}_{\bar 0}$ corresponds to the  BMS algebra, which is also called the Lie algebra $W(2,2)$ (\cite{ZD}). Additionally, the subalgebra ${\rm Vir} = \bigoplus_{n \in \mathbb{Z}} \mathbb{C} L_n \oplus \mathbb{C} \mathbf{c_1}$ represents the Virasoro algebra.

The  N=1 BMS superalgebra $\mathfrak{g}$ has a close relation with the N=1 Neveu-Schwarz algebra
$$
\mathfrak{ns}=\bigoplus_{n \in \mathbb{Z}} \mathbb{C} L_n \oplus \bigoplus_{r \in \mathbb{Z}+\frac{1}{2}} \mathbb{C} G_r \oplus \mathbb{C} \mathbf{c}
$$
which satisfies the following commutation relations:
$$
\begin{aligned}
{\left[L_m, L_n\right] } & =(m-n) L_{m+n}+\delta_{m,-n} \frac{m^3-m}{12} \mathbf{c}, \\
{\left[L_m, G_r\right] } & =\left(\frac{m}{2}-r\right) G_{m+r}, \\
{\left[G_r, G_s\right] } & =2 L_{r+s}+\frac{1}{3} \delta_{r+s, 0}\left(r^2-\frac{1}{4}\right) \mathbf{c}, \\
{[\mathfrak{ns}, \mathbf{c}] } & =0,
\end{aligned}
$$
for all $m, n \in \mathbb{Z}, r, s \in \mathbb{Z}+\frac{1}{2}$, where
$$
\left|L_n\right|=\overline{0}, \quad\left|G_r\right|=\overline{1}, \quad|\mathbf{c}|=\overline{0} .
$$

Let $A$ denote the truncated polynomial algebra $\mathbb{C}[t]/(t^3)$. Let $\mathfrak{ns}\otimes A$ denote the map Lie superalgebra with
$$
[x\otimes a, y\otimes b]=[x,y]\otimes ab,\quad x,y\in \mathfrak{ns},a,b\in A. 
$$
Then there is a natural embedding $\iota: \mathfrak{g}\to \mathfrak{ns}\otimes A$:

\begin{align*}
L_m &\mapsto L_m\otimes 1,\\
M_m &\mapsto L_m\otimes t^2,\\
Q_r &\mapsto G_r\otimes t,\\
\mathbf{c}_1&\mapsto \mathbf{c}\otimes 1,\\
\mathbf{c}_2&\mapsto \mathbf{c}\otimes t.
\end{align*}

		The  N=1 BMS superalgebra	$\frak g$ has an anti-involution\; $\bar{ }:\mathfrak{g}\to \mathfrak{g}$  given by
		\begin{eqnarray*}
			\overline{L_n} =L_{-n},\quad    \overline{M_n}=M_{-n},\quad    \overline{Q_r}=Q_{-r},\quad   \overline{{\bf c}_{1}}={\bf c}_{1},\quad  \overline{{\bf c}_{2}}={\bf c}_{2}
		\end{eqnarray*}
		for $n \in\z,r\in \z+\frac{1}{2}$. This anti-involution induces an anti-involution on the enveloping algebra $U(\frak g)$ naturally.

	Moreover,  $\frak g$  is equipped with a $\frac12\z$-grading by the eigenvalues of the adjoint action of $L_0$:
	$$
	\frak g=\bigoplus_{m\in\frac{1}{2}\z}\frak g_{m},
	$$ 
	where
	$$
	\frak g_m=\begin{cases}
		 \cc L_m\oplus \cc M_m \oplus\cc \delta_{m,0} {\bf c}_1\oplus\cc  \delta_{m,0}{\bf c}_2,\quad  &\text{if}\quad m\in\z;\\
		 \cc Q_m,\quad &\text{if}\quad m\in\z+\frac{1}{2}.
	\end{cases}
	$$
	It follows that $\frak g$  possesses a triangular decomposition:
	$
	\frak g=\frak g_{+}\oplus \frak g_{0}\oplus \frak g_{-},
	$
	where
	\begin{align*}
		\frak g_\pm&=\bigoplus_{n\in\z_+}\cc L_{\pm n}\oplus\bigoplus_{n\in\z_+}\cc M_{\pm n}\oplus  \bigoplus_{r\in\mathbb{N}+\frac{1}2}\cc Q_{\pm r},\\
		\frak g_{0}&=\cc L_0\oplus\cc M_0\oplus\cc {{\bf c}_1}\oplus\cc {{\bf c}_2}.
	\end{align*}

Let $\varphi: \mathfrak{g}_0\to \cc$ be a linear function.	
Recall that a highest weight vector  in a $U(\frak g)$-module $M$ is a vector $v\in M$ such that 
$$
xv=\varphi(x)v,\quad \mathfrak{g}_+v=0
$$ 
for all $x\in U(\frak g)$. A {\bf highest weight module}  is a $U(\frak g)$-module which is cyclically generated by a highest weight vector.

For any $k\in \mathbb{Z}_+$, let
$$
\frak g^{(k)}=\bigoplus_{i\ge k}\frak g_i\oplus \cc {\bf c}_1\oplus \cc {\bf c}_2.
$$ 
Then $\frak g^{(k)}$ is a subalgebra of $\frak g$. Let $\phi_k:\mathfrak{g}^{(k)}\to\mathbb C$ be  a nontrivial Lie superalgebra homomorphism. It follows that
$$\phi_k(L_i) = \phi_k(M_i)=0, \quad\forall i \ge 2k+1,\quad\text{and}\quad \phi_k(Q_{j+\frac12})=0,\quad \forall j\ge k.$$ 
A $\mathfrak{g}$-module $W$ is called a {\bf Whittaker module} if $W$ is generated by a vector $w\in W$, where $\mathbb Cw$ is a one-dimensional $\frak g^{(k)}$-module induced by $\phi_k$, i.e.
$xw=\phi_k(x)w$ for any $x\in \frak g^{(k)}$. In this case, $w$ is called a Whittaker vector. 

	A $\mathfrak{g}$-module $W$ is called  {\bf  smooth} in the sense that for every $w\in W$,
	$$L_iw=M_iw=Q_{i-\frac{1}{2}}w= 0$$
	for $i$ sufficiently large.
 
 It is clear that both highest weight modules  and Whittaker modules over $\mathfrak{g}$ are smooth.

If $W$ is a $\mathfrak{g}$-module on which  ${\bf c}_1, {\bf c}_2$ act as complex
	scalars $c_1, c_2$, respectively,   we say that $W$ is of {\bf  central charge } $(c_1, c_2)$. 
 For $c_1,c_2\in \mathbb{C}$, let $\mathcal{R}_{\mathfrak g}(c_1,c_2)$ denote the category  whose objects are  smooth $\mathfrak{g}$-modules of central charge $(c_1, c_2)$.


\begin{defi}\label{LF}
Let $L$ be a  Lie superalgebra, $V$ an $L$-module and $x\in L$.
\begin{itemize}
\item[\rm (1)] If for any $v\in V$ there exists $n\in{\mathbb Z}_+$ such that $x^nv=0$,  then the action of  $x$  on $V$ is said to be  {\bf locally nilpotent}.
\item[\rm (2)] If for any $v\in V$, $\mathrm{dim}\,(\sum_{n\in\mathbb N}\mathbb C x^nv)<+\infty$, then   the action of $x$ on $V$ is  said to be {\bf locally finite}.
\item[\rm (3)] The action of $L$   on $V$  is said to be locally nilpotent if for any $v\in V$ there exists an $n\in{\mathbb Z}_+$ (depending on $v$) such that $x_1x_2\cdots x_nv=0$ for any $x_1,x_2,\cdots, x_n\in L$.
\item[\rm (4)] The action of $L$  on $V$  is said to be locally finite if for any $v\in V$, {\rm dim}\,$U(L)v<\infty$.
\end{itemize}
\end{defi}

	Denote by  $\mathbb{M}$ the set of all infinite vectors of the form $\mi:=(\ldots, i_2, i_1)$ with entries in $\mathbb N$,
	satisfying the condition that the number of nonzero entries is finite, and $\mathbb{M}_1:=\{\mathbf{i}\in \mathbb{M}\mid i_k=0, 1, \ \forall k\in\z_+\}$.
	For $\mi\in\mathbb M$, denote by ${\rm supp}(\mi):=\{i_s\mid i_s\ne 0, s\in \mathbb Z_+\}$.

	Let $\mathbf{0}$ denote the element $(\ldots, 0, 0)\in\mathbb{M}$ and
	for
	$i\in\z_+$ let $\epsilon_i$ denote the element
	$$(\ldots,0,1,0,\ldots,0)\in\mathbb{M},$$
	where $1$ is
	in the $i$-th  position from right.  For ${\bf 0}\neq \mi\in\mathbb{M}$, let $p$ be the smallest integer such that $i_p \neq0$ and define  $\mi^\prime=\mi-\epsilon_p$.
	let $q$ be the largest integer such that $i_q \neq0$ and define  $\mi^{\prime\prime}=\mi-\epsilon_q$.

	For $\mi, \mj\in\mathbb{M}, \mk\in\mathbb{M}_1$, we denote
	$$\ell(\mi)=\sum i_k,$$
	$$Q^{\mk}M^{\mj}L^{\mi} =\ldots Q_{-2+\frac{1}{2}}^{k_2} Q_{-1+\frac{1}{2}}^{k_1}\ldots M_{-2}^{j_2}M_{-1}^{j_1}\ldots L_{-2}^{i_2} L_{-1}^{i_1}\in U(\frak g_{-}),$$
 and
	\begin{equation} {\rm w}(\mi, \mj, \mk)=\sum_{n\in\z_+}i_n\cdot n+\sum_{n\in\z_+}j_n\cdot n+\sum_{n\in\z_+}k_n\cdot (n-\frac12),\label{length}\end{equation} which is called the length of $(\mi, \mj, \mk)$ (or the length of $Q^{\mk}M^{\mj}L^{\mi}$).  For convenience, denote by $|\mi|= {\rm w}(\mi, \bf0, \bf0)$,  $|\mj|= {\rm w}(\bf0, \mj, \bf0)$ and $|\mk|= {\rm w}(\bf0, \bf0, \mk)$ for given $Q^\mk M^\mj L^\mi\in U(\frak g_-)$.

	\section{Verma modules over the N=1 BMS  superalgebra}

 In this section, we present a determinant formula for the contravariant form on the Verma modules over the N=1 BMS  superalgebra $\mathfrak{g}$ and establish a necessary and sufficient condition for the simplicity of the Verma modules.	

\subsection{Contravariant forms on Verma modules }

	For $h_1, h_2, {c_1, c_2}\in\cc$, let $\cc $  be a one-dimensional  ${\frak g}_{0}$-module  defined by
	$$ L_01=h_11, \quad M_01=h_2 1,\quad   {\bf c}_11=c_11,\quad  {\bf c}_21=c_21.$$
	Let ${\frak g}_{+}$ act trivially on $1$, making $\mathbb{C} $ be a $({\frak g}_{0}\oplus{\frak g}_{+})$-module.
	The Verma module $M_{\frak g}(h_1, h_2, c_1, c_2)$ over $\frak g$ is defined by
	$$ M_{\frak g}(h_1, h_2, c_1, c_2)=U(\frak g)\otimes_{U(\frak g_{0}\oplus\frak g_{+})}\cc\cong U(\frak g_{-}){\bf 1},\quad (\text{as vector spaces}) ,$$
	where ${\bf 1}=1\otimes 1$. Moreover,
	$M_{\frak g}(h_1, h_2, c_1, c_2)=\bigoplus_{n\in\frac{1}{2}\mathbb{N}}M_{\frak g}(h_1, h_2, c_1, c_2)_n,$
	where
	$$M_{\frak g}(h_1, h_2, c_1, c_2)_n
	=\{v \in M_{\frak g}(h_1, h_2, c_1, c_2)\,|\,L_0v =(h_1+n)v\}.
	$$
	
	According to the Poincar\'{e}-Birkhoff-Witt ($\mathrm{PBW}$)  theorem, every element $v$ of $M_{\frak g}(h_1, h_2, c_1, c_2)$ can be uniquely written in
	the following form
	\begin{equation}\label{def2.1}
		v=\sum_{\mi, \mj\in\mathbb{M}, \mk\in\mathbb{M}_1}a_{\mi, \mj, \mk}Q^{\mk}M^{\mj}L^{\mi} {\bf 1},
	\end{equation}
	where $a_{\mi, \mj, \mk}\in\mathbb C$ and only finitely many of them are nonzero. For any $v\in M_{\frak g}(h_1, h_2, c_1, c_2)$ as in  \eqref{def2.1}, we denote by $\mathrm{supp}(v)$ the set of all $(\mi, \mj, \mk)\in \mathbb M_1\times\mathbb{M}^2$  such that $a_{\mi, \mj, \mk}\neq0$.	

	The Verma module $M_{\frak g}(h_1, h_2, c_1, c_2)$  has
	a unique maximal submodule $J(h_1, h_2, c_1, c_2)$ and the factor module
	$$
	L_{\frak g}(h_1, h_2, c_1, c_2)=M_{\frak g}(h_1, h_2, c_1, c_2)/J(h_1, h_2, c_1, c_2)
	$$
	is simple.

	Let $\langle \cdot ,\cdot \rangle$ be a $\cc$-value bilinear form  on $M_{\frak g}(h_1, h_2, c_1, c_2)$ defined by
	\begin{eqnarray*}
		\langle a{\bf 1} , b{\bf 1} \rangle=\langle {\bf 1} , {\rm P}(\bar a b){\bf 1}\rangle ,\quad  \langle {\bf 1},  {\bf 1}\rangle=1,\quad\forall\  a,b\in U(\frak g_{-} ),
	\end{eqnarray*} 
	where ${\rm P}:U(\frak g)\to U(\frak g_{0})$ be the Harish-Chandra projection, i.e.,
	a projection along the decomposition
	\begin{eqnarray*}
		&&U(\frak g)= U(\frak g_{0}) \oplus (\frak g_{-} U(\frak g) +U(\frak g)\frak g_{+}).
	\end{eqnarray*}
	Since ${\rm P}(\overline{a}b)={\rm P}(\overline{b}a)$, we know that the bilinear form $\langle \cdot ,\cdot \rangle$ is symmetric, i.e.
	\begin{equation}\label{symm}
	\langle a{\bf 1} , b{\bf 1} \rangle=\langle b{\bf 1}, a{\bf 1}\rangle,\quad\forall\  a,b\in U(\frak g_{-} ).
	\end{equation}
Also,  the form $\langle\cdot ,\cdot \rangle$ is contravariant, i.e.,
for $m\in\z, r\in\z+\frac{1}{2}$, $u,v\in M_{\frak g}(h_1, h_2, c_1, c_2)$, we have 	$$
	\langle L_m u, v\rangle=\langle  u, L_{-m}v\rangle,\quad \langle M_m u, v\rangle=\langle  u, M_{-m}v\rangle,\quad\langle Q_r u, v\rangle=\langle  u, Q_{-r}v\rangle.
	$$
	   The simplicity of the Verma module $M_{\frak g}(h_1, h_2, c_1, c_2)$ is encoded in the determinant of the contravariant form. Since distinct weight spaces of $M_{\frak g}(h_1, h_2, c_1, c_2)$ are orthogonal with respect to $\langle\;,\;\rangle$, the Verma module $M_{\frak g}(h_1, h_2, c_1, c_2)$ is simple if and only if the restricted forms
	$$\langle\cdot,\cdot\rangle_n: M_{\frak g}(h_1, h_2, c_1, c_2)_n\times M_{\frak g}(h_1, h_2, c_1, c_2)_n\to \cc$$
	are nondegenerate for all $n\in\frac{1}{2}\mathbb{N}$.
	
	For $n\in\frac{1}{2}\mathbb{N}$, let
	$$
	S_n=\{Q^{\mk}M^{\mj}L^{\mi}{\bf 1} : {\rm w}(\mi, \mj, \mk)=n, \forall\  \mi, \mj\in\mathbb M, \mk\in\mathbb M_1\}.
	$$
	From the PBW theorem, $S_n$ is a basis of $M_{\frak g}(h_1, h_2, c_1, c_2)_n$. Hence
	$$
	|S_n|=p(n):=\#\left\{(\mi,\mj,\mk)\in  \mathbb M^2\times\mathbb M_1 : {\rm w}(\mi, \mj, \mk)=n\right\},
	$$
	where $p(n)$ can be counted by the generating series
	\begin{eqnarray*}
		\sum_{n\in\frac{1}{2}\mathbb{N}} p(n)q^n=\frac{\prod_{k\in\mathbb{N}+\frac{1}{2}}(1+q^k)}{\prod_{k\geq1}(1-q^k)^{2}}.
	\end{eqnarray*}

	Denote by $>$ the  {\em  lexicographical  total order}  on  $\mathbb{M}$,  defined as follows: for any $\mi,\mj\in\mathbb{M}$, set $$\mi> \mathbf{j}\ \Leftrightarrow \ \mathrm{ there\ exists} \ r\in\mathbb Z_+ \ \mathrm{such \ that} \ i_r>j_r\ \mathrm{and} \ i_s=j_s,\ \forall\; s>r.$$

Now we can induce a  principal total order $>$ on  $S_{n}$:	$$Q^{\mk}M^{\mj}L^{\mi} {\bf 1}>Q^{\mk_1}M^{\mj_1}L^{\mi_1}{\bf 1}$$
	if and only if one of the following conditions is satisfied:
	\begin{itemize}
		\item[(i)]$|\mj|>|\mj_1|;$
		\item[(ii)]$|\mj|=|\mj_1|\ \mbox{and}\
		\mj_1>\mj$;
		\item[(iii)]$\mj=\mj_1,\ |\mk|>|\mk_1|;$
		\item[(iv)]
		$\mj=\mj_1,\ |\mk|=|\mk_1|\ \mbox{and}\
		\mk>\mk_1;$
		\item[(v)]$\mj=\mj_1,\ \mk=\mk_1,\ \mbox{and}\
		\mi>\mi_1.$
	\end{itemize}

	Sort the elements in $S_n$ by $>$, for example,
	\begin{align*}
		S_0&=\{{\bf 1}\},\\
		S_{\frac{1}{2}}&=\{Q_{-\frac{1}{2}}{\bf 1}\},\\
		S_{1}&=\{M_{-1}{\bf 1}, L_{-1}{\bf 1}\},\\
		S_{\frac{3}{2}}&=\left\{Q_{-\frac{1}{2}}M_{-1}{\bf 1},Q_{-\frac{3}{2}}{\bf 1}, Q_{-\frac{1}{2}}L_{-1}{\bf 1}\right\},\\
		S_{2}&=\left\{M_{-1}^2{\bf 1},M_{-2}{\bf 1}, M_{-1}L_{-1}{\bf 1},Q_{-\frac{3}{2}}Q_{-\frac{1}{2}}{\bf 1}, L_{-2}{\bf 1},L_{-1}^2{\bf 1}\right\}.
	\end{align*}
	

\subsection{Simplicity of Verma modules }
 
 Let $G_n$ be the Gram matrix of the form  $\langle\cdot ,\cdot \rangle_n$ defined by
	$$
	\langle Q^{\mk}M^{\mj}L^{\mi}{\bf 1}, Q^{\mk_1}M^{\mj_1}L^{\mi_1}{\bf 1}\rangle,
	$$ where ${\rm w}(\mi, \mj, \mk)={\rm w}(\mi_1, \mj_1, \mk_1)=n$.
	
	For $u=Q^{\mk}M^{\mj}L^{\mi}{\bf 1}\in S_n,$
	we define
	\begin{eqnarray*}
		u^*=Q^{\mk}M^{\mi}L^{\mj}{\bf 1}.
	\end{eqnarray*}
Then
	$
	S_{n}^*=\{u^*\mid u\in S_n\}
	$
	is also a basis of $M_{\frak g}(h_1, h_2, c_1, c_2)_n$.  It is clear that
	\begin{align*}
		S^*_0&=\{{\bf 1}\},\\
		S^*_{\frac{1}{2}}&=\{Q_{-\frac{1}{2}}{\bf 1}\},\\
		S^*_{1}&=\{L_{-1}{\bf 1},M_{-1}{\bf 1}\},\\
		S^*_{\frac{3}{2}}&=\left\{ Q_{-\frac{1}{2}}L_{-1}{\bf 1},Q_{-\frac{3}{2}}{\bf 1},Q_{-\frac{1}{2}}M_{-1}{\bf 1}\right\},\\
		S^*_{2}&=\left\{L_{-1}^2{\bf 1},L_{-2}{\bf 1},L_{-1}M_{-1}{\bf 1},Q_{-\frac{3}{2}}Q_{-\frac{1}{2}}{\bf 1},M_{-2}{\bf 1},  M_{-1}^2{\bf 1}\right\}.
	\end{align*}

	

	Let $D_n=(d_{ij})$, where $d_{ij}=\langle b_{i},b_{j}^* \rangle$ for  $i,j=1,\dots p(n)$. Then
	$$D_{\frac12}=(2h_2),\quad D_1=
	\begin{pmatrix}
		2h_2 & 0\\
		2h_1 & 2h_2\\
	\end{pmatrix},\quad D_{\frac{3}{2}}=
	\begin{pmatrix}
		4h_2^2 & 0&0\\
		* & 2h_2+\frac{2}{3}c_2&0\\
		* &*&4h_2^2\\
	\end{pmatrix},
	$$
	$$D_{2}=
	\begin{pmatrix}
		8h_2^2 & 0&0&0&0&0\\
		* & 4h_2+\frac{1}{2}c_2&0&0&0&0\\
		* &*&4h_2^2&0&0&0\\
		* &*&*&2h_2(2h_2+\frac{2}{3}c_2)&0&0\\
		* &*&*&*&4h_2+\frac{1}{2}c_2&0\\
		* &*&*&*&*&8h_2^2\\
	\end{pmatrix}.
	$$

	If $\mj\in\mathbb M$, $m\in\mathbb Z_+$ with $m\ge k$ for all $k\in{\rm supp}(\mj)$, then
	$$
	L_mM^\mj{\bf 1}=\left(2mh_2+\frac{m^3-m}{12}c_2\right)\frac{\partial }{\partial M_{-m}}M^\mj{\bf 1},
	$$
	where $\frac{\partial }{\partial M_{-m}}$ is  the operation of formal partial derivative defined by  the Leibniz rule  and
	$$\frac{\partial }{\partial M_{-m}}(M_{-n})=\delta_{m,n},\quad \frac{\partial }{\partial M_{-m}}({\bf 1})=0.$$
		It follows  immediately that if $\mi >\mj$,
	\begin{equation}\label{LM}
		\left\langle  L^{\mi}{\bf 1}, M^{\mj} {\bf 1}\right\rangle =L_{1}^{i_1}L_2^{i_2}\cdots M_{-2}^{j_2}M_{-1}^{j_1}{\bf 1}=0.
	\end{equation}
	
	If $\mk\in \mathbb M_1$, $t\in \mathbb{Z}_+$ and $k_t=1, k_{t+i}=0$, then
	$$
	Q_{t-\frac{1}{2}}Q^\mk{\bf 1}=\left(2h_2+\frac{4t^2-1}{12}c_2\right)Q^{\mk''}{\bf 1}.
	$$
	It follows  immediately that if $\mi>\mi_1$,
	\begin{equation}\label{QQ}
		\left\langle  Q^{\mi}{\bf 1},Q^{\mi_1} {\bf 1}\right\rangle =0.
	\end{equation}

	\begin{lemm}\label{l3.1}
		The matrix $D_n$ is lower triangular for any $n\in\frac12 \z_+$. Moreover, the diagonal entries of $D_n$ are
		$$d_{ii}=\prod_{r\in {\rm supp}(\mi)}r\left(2h_2+\frac{r^2-1}{12}c_2\right)^{i_r}\prod_{t\in{\rm supp}(\mk)}\left(2h_2+\frac{4t^2-1}{12}c_2\right)^{k_t}
		\prod_{s\in {\rm supp}(\mj) } s\left(2h_2+\frac{s^2-1}{12}c_2\right)^{j_s},
		$$
		where $b_i=Q^{\mk}M^{\mj}L^{\mi}{\bf 1}\in S_n$.
	\end{lemm}
	\begin{proof}
		Suppose that		$b_i=Q^{\mk}M^{\mj}L^{\mi}{\bf 1} >Q^{\mk_1}M^{\mj_1}L^{\mi_1}{\bf 1}=b_j $ with $Q^{\mk}M^{\mj}L^{\mi}{\bf 1},Q^{\mk_1}M^{\mj_1}L^{\mi_1}{\bf 1}\in S_n.$ It suffices to show
  $$\left\langle b_i, b_j^*\right\rangle=
\left\langle Q^{\mk}M^{\mj}L^{\mi}{\bf 1}, \left(Q^{\mk_1}M^{\mj_1}L^{\mi_1}{\bf 1}\right)^*\right\rangle=0.
  $$
		\begin{itemize}
			\item If $|\mj|>|\mj_1|$, then
			$$0=\left\langle M^\mj{\bf 1}, L^{\mj_1}{\bf 1} \right\rangle=\left\langle {\bf 1}, \overline{M^\mj}L^{\mj_1}{\bf 1} \right\rangle \Rightarrow \overline{M^\mj}L^{\mj_1}{\bf 1}=0.$$
			It follows that
			\begin{align*}
				d_{ij}&=\left\langle  Q^\mk M^\mj L^\mi{\bf 1}, Q^{\mk_1}M^{\mi_1}L^{\mj_1}{\bf 1} \right\rangle\\
				&=\left\langle {\bf 1}, \overline{L^\mi}\cdot \overline{M^\mj}\cdot \overline{Q^\mk}\cdot Q^{\mk_1}M^{\mi_1}L^{\mj_1}{\bf 1} \right\rangle\\
				&=\left\langle {\bf 1}, \overline{L^\mi}\cdot\overline{Q^\mk}\cdot Q^{\mk_1}M^{\mi_1}\overline{M^\mj}L^{\mj_1}{\bf 1} \right\rangle\\
				&=0.
			\end{align*}
			
			
			
			\item If $|\mj|=|\mj_1|\ \mbox{and}\
			\mj<\mj_1$, it follows from \eqref{symm} and (\ref{LM}) that $\overline{M^\mj}L^{\mj_1}{\bf 1}=\overline{L^{\mj_1}}M^\mj{\bf 1}=0$. Then
			\begin{align*}
				d_{ij}&=\left\langle Q^\mk M^\mj L^\mi{\bf 1},Q^{\mk_1}M^{\mi_1}L^{\mj_1}{\bf 1} \right\rangle\\
				&=\left\langle {\bf 1}, \overline{L^\mi}\cdot \overline{M^\mj}\cdot \overline{Q^\mk}\cdot Q^{\mk_1}M^{\mi_1}L^{\mj_1}{\bf 1} \right\rangle\\
    &=\left\langle {\bf 1}, \overline{L^\mi}\cdot \overline{Q^\mk}\cdot Q^{\mk_1}M^{\mi_1}\overline{M^\mj}L^{\mj_1}{\bf 1} \right\rangle\\
				&=0.
			\end{align*}
			
			
			\item If
			$\mj=\mj_1,\ |\mk|>|\mk_1|$ or  $|\mk|=|\mk_1|, \mk>\mk_1$,  it is clear that $\overline{Q^\mk}Q^{\mk_1}{\bf 1} =0$. It follows that
			\begin{align*}
				d_{ij}&=\left\langle Q^\mk M^\mj L^\mi{\bf 1},Q^{\mk_1}M^{\mi_1}L^{\mj_1}{\bf 1} \right\rangle\\
				&=\left\langle {\bf 1}, \overline{L^{\mi}}M^{\mi_1}\overline{Q^{\mk}}Q^{\mk_1}{\bf 1} \right\rangle\left\langle  {\bf 1}, \overline{M^\mj} L^{\mj_1}{\bf 1} \right\rangle\\
				&=0.
			\end{align*}
			
			\item If $\mj=\mj_1,\ \mk=\mk_1,\ \mbox{and}\
			\mi>\mi_1$,  it follows from (\ref{LM}) that $\overline{L^\mi}M^{\mi_1} {\bf 1} =0$. Then
			\begin{align*}
				d_{ij}&=\left\langle Q^\mk M^\mj L^\mi{\bf 1},Q^{\mk_1}M^{\mi_1}L^{\mj_1}{\bf 1} \right\rangle\\
				&=\left\langle {\bf 1}, \overline{L^\mi}\cdot \overline{M^\mj}\cdot \overline{Q^\mk}\cdot Q^{\mk_1}M^{\mi_1}L^{\mj_1}{\bf 1} \right\rangle\\
				&=\left\langle {\bf 1}, \overline{L^\mi} M^{\mi_1}{\bf 1}\right\rangle \left\langle {\bf 1},\overline{Q^\mk} Q^{\mk_1}{\bf 1} \right\rangle \left\langle  {\bf 1}, \overline{M^\mj} L^{\mj_1}{\bf 1} \right\rangle\\
				&=0.
			\end{align*}
		\end{itemize}

		Moreover,
		\begin{align*}
			d_{ii}&=\left\langle {\bf 1}, \overline{L^\mi} M^{\mi}\right\rangle \left\langle {\bf 1},\overline{Q^\mk} Q^{\mk}{\bf 1} \right\rangle \left\langle  {\bf 1}, \overline{M^\mj} L^{\mj}{\bf 1} \right\rangle\\
			&=\prod_{r\in {\rm supp}(\mi)}r\left(2h_2+\frac{r^2-1}{12}c_2\right)^{i_r}\prod_{t\in{\rm supp}(\mk)}\left(2h_2+\frac{4t^2-1}{12}c_2\right)^{k_t}
			\prod_{s\in {\rm supp}(\mj) } s\left(2h_2+\frac{s^2-1}{12}c_2\right)^{j_s}
		\end{align*}
		for $i=1,2,\dots p(n)$.
	\end{proof}

 We are now in the position to state the main result in this section:
	
	\begin{theo}\label{main1}The Verma module $M_{\frak g}(h_1, h_2, c_1, c_2)$  is simple if and only if
		\begin{equation*}
			h_{2}+\frac{i^2-1}{24}c_{2}\neq 0,\quad \forall \ i\in \z_+.
		\end{equation*}
	\end{theo}
	\begin{proof}For all $n\in\frac{1}{2}\mathbb{N}$, the restricted form
	$$\langle\cdot,\cdot\rangle_n: M_{\frak g}(h_1, h_2, c_1, c_2)_n\times M_{\frak g}(h_1, h_2, c_1, c_2)_n\to \cc$$
is nondegenerate if and only if the determinant of the Gram matrix $G_n$ is nonzero,  if and only if ${\rm det}\,D_n\neq 0$. By applying Lemma \ref{l3.1}, we can directly conclude that ${\rm det}\,D_n=\prod_{i=1}^{p(n)}d_{ii}\neq 0$ if and only if
\begin{equation*}
h_{2}+\frac{i^2-1}{24}c_{2}\neq 0, \quad \forall \ i\in \z_+.
\end{equation*}
\end{proof}

As a corollary of Theorem \ref{main1}, we know that $M_{\frak g}(0,0,c_1,c_2)$ is not simple.  In fact, $ Q_{-\frac{1}{2}}{\bf 1}$ is a singular vector of $M(0,0,c_1,c_2)$, and then $L_{-1}{\bf 1}$ is a subsingular vector. 
	The {\it vaccum module} for $\frak g$ is
	the quotient module
	$$
	V_{\frak g}(c_1,c_2)=M_{\frak g}(0,0,c_1,c_2)/J,
	$$
	where $J$ is the submodule of $V_{\frak g}(0,0,c_1,c_2)$ generated by  $L_{-1}{\bf 1}$. Denote by ${\mathds{1} }:={\bf 1}+J$  the vacuum vector of $V_{\frak g}(c_1,c_2)$. Using a similar argument as Theorem \ref{main1}, 	one obtains the following criterion of
	simplicity of the vacuum module:
	\begin{theo}
		The  vaccum module $V_{\frak g}(c_1,c_2)$  is simple if and only if    $c_2\neq 0$.  
	\end{theo}

\section{Classifications of simple smooth $\mathfrak{g}$-modules}

	In this section, we construct and classify all simple smooth $\mathfrak{g}$-modules of central charge $(c_1,c_2)$   under certain conditions.

	\subsection{Constructing simple smooth $\mathfrak{g}$-modules}

	For any $m\in\mathbb N, n\in\mathbb Z$ with $n\le 1$, let
	\begin{eqnarray*}
		\frak g^{(m, n)}&=&\bigoplus_{i\geq m}\cc L_i\oplus\bigoplus_{j\geq n}\cc M_j\oplus\bigoplus_{k\ge 1}\cc Q_{k-\frac12}\oplus\mathbb C{\bf c}_1\oplus\mathbb C{\bf c}_2.
	\end{eqnarray*}
	Note that $\frak g^{(m, n)}$ is a subalgebra of $\frak g$  and $\mathfrak g^{(0, 0)}=\frak g_+\oplus\frak g_0$.

	Given $q\in\mathbb N$ and 
	a $\frak g^{(0, -q)}$-module $V$ of  central charge   $(c_1, c_2)$.  We consider the following  induced $\frak g$-module
	$$
	{\rm Ind}_q(V)={\rm Ind}^{\mathfrak g}_{\mathfrak g^{(0, -q)}} V.
	$$
 According to the {PBW}  theorem, every element $v$ of ${\rm Ind}_q(V)$ can be uniquely written in the following form
	\begin{equation}\label{def4.1}
		v=\sum_{\mi, \mj\in\mathbb{M}, \mk\in\mathbb{M}_1}Q^{\mk}M^{\mj}L^{\mi} v_{\mi, \mj, \mk},
	\end{equation}
	where all  $v_{\mi, \mj,\mk}\in V$ and only finitely many of them are nonzero.


	Denote by $\prec$ the  {\em reverse  lexicographical  total order}  on  $\mathbb{M}$,  defined as follows: for any $\mi,\mk\in\mathbb{M}$, set $$\mi\prec \mathbf{j}\ \Leftrightarrow \ \mathrm{ there\ exists} \ r\in\z_+ \ \mathrm{such \ that} \ i_r <j_r\ \mathrm{and} \ i_s=j_s,\ \forall\; 1\le s<r.$$

	Next we introduce a principal total order on $\mathbb{M}\times\mathbb{M}_1$ as follows.

	Denoted by $\prec$:
	$(\mj,\mk) \prec (\mj_1,\mk_1)$ if and only if one of the following conditions is satisfied:

	(1)  $\mk<\mk_1$;
	
	(2) $\ \mk=\mk_1\ \mathrm{and }\ \mj < \mj_1$.

	Now we can induce a  principal total order on $\mathbb{M}^2\times\mathbb{M}_1$, still denoted by $\prec$:
	$(\mi, \mj, \mk) \prec (\mi_1, \mj_1, \mk_1) $ if and only if one of the following conditions is satisfied:

	(1) ${\rm w}(\mi, \mj, \mk)<{\rm w}(\mi_1, \mj_1, \mk_1)$;
	
	(2) ${\rm w}(\mi, \mj, \mk)={\rm w}(\mi_1, \mj_1, \mk_1)$  and $\mi_1\prec\mi$;
	
	(3) ${\rm w}(\mi, \mj, \mk)={\rm w}(\mi_1, \mj_1, \mk_1)$ and $\mi =\mathbf{i}_1$ and $(\mj, \mk) \prec (\mj_1, \mk_1)$.

 For any $v\in{\rm Ind}_q(V)$ as in  \eqref{def2.1}, we denote by $\mathrm{supp}(v)$ the set of all $(\mi, \mj, \mk)\in \mathbb{M}^2\times\mathbb{M}_1$  such that $v_{\mi, \mj, \mk}\neq0$.
 For a nonzero $v\in {\rm Ind}_q(V)$, we write $\mathrm{deg}(v)$  the maximal (with respect to the principal total order on $\mathbb{M}^2\times\mathbb{M}_1$) element in $\mathrm{supp}(v)$, called the { degree} of $v$. Note that here and later
 we make the
 convention that  $\mathrm{deg}(v)$ is only for  $v\neq0$.

\begin{theo}\label{main3}
For $q\in \mathbb{N}$, $c_1,c_2\in \mathbb{C}$, let $V$ be
	a simple $\frak g^{(0, -q)}$-module  of  central charge   $(c_1, c_2)$.
Assume that there exists  $t\in\mathbb N$ satisfying the following two conditions:
		\begin{itemize}
			\item[{\rm (a)}] the action of $M_t$ on $V$ is injective if $t>0$ or the action of $M_0+\frac{1}{24}(n^2-1){\bf c}_2$ on $V$ is injective for all $n\in \mathbb{Z}^*$ if $t=0$;

			\item[{\rm (b)}] $M_iV=0$ for all $i>t$  and $L_jV=0$ for all $j>t+q$.
		\end{itemize}
		Then
		\begin{itemize}
			\item[{\rm (i)}]  $Q_{i-\frac12}V=0$ for all $i>t$;
			\item[{\rm (ii)}] The induced module ${\rm Ind}_q(V)$ is a simple  module in $\mathcal{R}_{\mathfrak g}(c_1,c_2)$.
		\end{itemize}
\end{theo}
\begin{proof} 
	
(i) For $i>t\ge0$,  we have $Q_{i-\frac12}^2V=M_{2i-1}V=0$ by (b). If $Q_{j-\frac12}V=0$ we are done. Otherwise,   $W=Q_{j-\frac12}V$    is a proper subspace of $V$.   For $j\in\mathbb Z_+$, we have
		$$
		Q_{j-\frac{1}{2}}W=Q_{j-\frac{1}{2}}Q_{i-\frac12}V=2M_{i+j-1}V-Q_{i-\frac12}Q_{j-\frac{1}{2}}V\subset Q_{i-\frac12}V=W.
		$$
		Clearly $L_nW\subset W$ and $M_mW\subset W$ for any $n\in\mathbb N$ and $m\in \z$ with $ m\ge -q$.
		It follows that $W$ is a proper $\frak g^{(0, -q)}$-submodule of $V$. Then $W=Q_{i-\frac12}V=0$ for $i>t$ since $V$ is simple.

		(ii)
 Since  the actions of  $L_i, M_i, Q_{i-\frac12}$ on ${\rm Ind}_q(V)$  for  all $i>t+q$  are locally nilpotent, 
  ${\rm Ind}_q(V)$ is an object in $\mathcal{R}_{\mathfrak g}(c_1,c_2)$.

Suppose that $W$ is a nonzero proper $\mathfrak g$-submodule of ${\rm Ind}_q(V)$. Now we are going to deduce some contradictions. Take an element $v\in W\setminus V $ such that  $\mathrm{deg}(v)=(\bmi,\bmj,\bmk)$ is minimal. Write
		\begin{equation}
			v=\sum_{\mi, \mj\in\mathbb{M},\, \mk\in\mathbb{M}_1}Q^{\mk}M^{\mj}L^{\mi}v_{\mi,\mj, \mk},
		\end{equation}
		where all  $v_{\mi, \mj, \mk}\in V$ and only finitely many of them are nonzero. 

Note that there exists  $t\in\mathbb N$ satisfying the following condition:
the action of $M_t$ on $V$ is injective if $t>0$ or the action of $M_0+\frac{1}{24}(n^2-1){\bf c}_2$ on $V$ is injective for all $n\in \mathbb{Z}^*$ if $t=0$.

\noindent{\bf Claim 1.} $\bmi=\mathbf{0}$.

If $\bmi\neq\mathbf{0}$, then $\hat{\underline{i}}={\rm min}\{s\mid i_s\ne0\}>0$, we shall prove that $\mathrm{deg}(M_{{\hat{\underline i}}+t}v)\prec (\bmi, \bmj, \bmk)$. It contradicts to the choice of $v$. Note that $M_{\hat{\underline i}+t}v_{\mi, \mj, \mk}=0$ for any  $(\mi, \mj, \mk)\in \mathrm{supp}(v)$.

	{\bf Case 1.} $t>0$. 	
		One can easily check that
		$$\deg(M_{\hat{\underline i}+t}Q^{\bmk}M^{\bmj}L^{\bmi}v_{\bmi, \bmj, \bmk})=\deg(Q^{\mathbf{k}}M^{\bmj}[M_{\hat{\underline i}+t}, L^{\bmi}]v_{\bmi, \bmj, \mk})=(\bmi^\prime, \bmj, \bmk).$$
		Here $M_{t}v_{\bmi, \bmj, \bmk}\neq0$ by (a).
		
		For $(\mi, \mj, \mk)\in{\rm supp}(v)$ with $(\mi, \mj, \mk)\prec (\bmi, \bmj, \bmk)$,  if
		$
		{\rm w}(\mi, \mj, \mk)<{\rm w}(\bmi, \bmj, \bmk),
		$
		then
		$$
		\deg M_{\hat{\underline i}+t}Q^{\mk}M^{\mj}L^{\mi}v_{\mi, \mj, \mk}\prec (\bmi^{\prime}, \bmj, \bmk).
		$$
		
		Now we suppose that ${\rm w}(\mi, \mj, \mk)={\rm w}(\bmi, \bmj, \bmk)$
		and $ \mathbf{i} \prec \bmi$.

		If $\hat{i}>\hat{\underline i}$, then ${\rm w}(\mi^\prime, \mj, \mk)<{\rm w}(\bmi^\prime, \bmj, \bmk)$. Hence
		$$\mathrm{deg}(M_{\hat{\underline i}+t}Q^{\mk}M^{\mj}L^{\mi}v_{\mi, \mj, \mk})
		\prec(\mi^\prime, \mj, \mk)\prec(\bmi^\prime, \bmj, \bmk).$$
		
  If $\hat{i}=\hat{\underline i}$, then $$\mathrm{deg}(M_{\hat{\underline i}+t}Q^{\mk}M^{\mj}L^{\mi}v_{\mi, \mj, \mk})
		=(\mi^\prime, \mj, \mk)\prec(\bmi^\prime, \bmj, \bmk).$$

		If ${\rm w}(\mi, \mj, \mk)={\rm w}(\bmi, \bmj, \bmk)$ and $\mathbf{i}=\bmi$, it is easy to see that
		$$\deg(M_{\hat{\underline i}+t}Q^{\mk}M^{\mj}L^{\mi}v_{\mi, \mj, \mk})=(\mi^\prime, \mj, \mk)\preceq(\bmi^\prime, \bmj, \bmk).$$
		 So Claim 1 holds in this case.
			
{\bf Case 2}: $t=0$.	  

Set $|v|_{\rm Vir}={\rm max}\{|\mi|: (\mi, \mj, \mk)\in{\rm supp}(v)\}$ and $\ell={\rm max}\{\ell(\mi): (\mi, \mj, \mk)\in{\rm supp}(v), |\mi|=|v|_{\rm Vir}\}$. Now 
we rewrite $v$ as  
\begin{equation}v=\sum_{|\mi|=|v|_{\rm Vir}}L^{\mi}u_{\mi}+u,\label{equ-u}\end{equation} where $|u|_{\rm Vir}<|v|_{\rm Vir}$ and $u_{\mi}\in U(\frak b)V$, $\frak b$ is the subalgebra of $\frak g$ generated by $\{M_i, Q_{i-\frac12}, {\bf c}_2: i\in\mathbb Z\}$. Note that $M_iu_{\mi}=0$ for any $i\in\mathbb Z_+$.

For $\mi=(\cdots, i_2, i_1)\in\mathbb M$, denote by $$
(L^{\mi})^{\#}:=M_{1}^{i_1}M_{2}^{i_2}\cdots.$$   Set
$$\mathbb S:=\{\mi: (\mi, \mj, \mk)\in{\rm supp}(v), |\mi|=|v|_{\rm Vir},  \ell(\mi)=\ell\}.
$$
Choose the maximal ${\mi}_{\rm max}\in \mathbb S$ with respect to ``$\prec$".
By action of $(L^{{\mi}_{\rm max}})^{\#}$ on \eqref{equ-u}, we get 
$$(L^{{\mi}_{\rm max}})^{\#}v=(L^{{\mi}_{\rm max}})^{\#}\cdot L^{{\mi}_{\rm max}}u_{{\mi}_{\rm max}}\ne0 $$ 
since $M_0+\frac{1}{24}(n^2-1){\bf c}_2$ on $V$ is injective for all $n\in \mathbb{Z}^*$, and  $(L^{{\mi}_{\rm max}})^{\#}\cdot L^{\mi}u_{\mi}=0$ if $\mi\prec {\mi}_{\rm max}$ and $(L^{{\mi}_{\rm max}})^{\#}u=0$ for $u$ in \eqref{equ-u} (see \cite[Theorem 5.2]{LPXZ}). This is a contradiction to the choice of $v$.

		So Claim 1 holds in all cases.
		Since $[M_i, Q^{\mk}M^{\mj}]=0$ for any $i\in\mathbb Z, \mj\in\mathbb M, \mk\in\mathbb M_1$,  repeating the above steps, we can suppose that 
		\begin{equation}
			v=\sum_{(\mj, \mk)\in \mathbb M\times\mathbb M_1} Q^{\mk}M^{\mj}v_{\mj, \mk},
		\end{equation}
		where all $v_{\mj, \mk}\in V$ and only finitely many of them are nonzero.
		Note that $Q_{\hat{\hat{\underline k}}+t-\frac12}V=0$.

	{\bf Claim 2.} $\bmk=\mathbf{0}$.

		If $\bmk\ne 0$, one can easily check that
		\begin{equation}\deg(Q_{\hat{\hat{\underline k}}+t-\frac12}v)=\deg\left(\sum_{(\mj, \mk)\in \mathbb M\times\mathbb M_1}Q_{\hat{\hat{\underline k}}+t-\frac12}Q^{\mk}M^{\mj}v_{\mj, \mk}\right)=\deg\left(\sum_{(\mj, \mk)\in \mathbb M\times\mathbb M_1}Q^{\mathbf{k}''}M^{\mathbf{j}}v_{\mj, \mk}\right)\prec(\mathbf 0, \bmj, \bmk).\label{maxlength}\end{equation}
			It also contradicts the choice of $v$.	
		
		   Next suppose that
		\begin{equation}
			v=\sum_{\mj\in\mathbb M} M^{\mj}v_{\mj},
		\end{equation}
		where all  $v_\mj\in V$ and only finitely many of them are nonzero. If $\bmj\ne0$, one can easily check that
		$$\mathrm{deg}(L_{\hat{\hat j}+t}v)\prec(\bf 0, \bmj, \bf 0).$$		
		It also contradicts the choice of $v$. 
  Then the theorem holds.
\end{proof}

	\begin{rema}			
				If $t=q=0$ in Theorem \ref{main3}, then $\frak{g}_+V=0$. Since $\frak{g}^{(0,0)}/\mathfrak{g}_+$ is abelian, $V$ has to be one-dimensional and then
 $\mathrm{Ind}_0(V)$ is a highest weight module over $\frak{g}$.

	\end{rema}

	\subsection{Characterization of simple smooth $\mathfrak{g}$-modules}
	
	In this subsection, we give a characterization of simple $\frak g$-modules in $\mathcal{R}_{\mathfrak g}(c_1,c_2)$.

	\begin{prop}\label{th2}
		Let $W$ be a simple $\frak g$-module in $\mathcal{R}_{\mathfrak g}(c_1,c_2)$ such that the action of  $M_0+\frac1{24}(n^2-1){\bf c}_2$  is injective on $W$  for any $n\in\mathbb Z^*$. Then the following statements are equivalent:
		
		\begin{itemize}
			\item[\rm(1)]   There exists $l\in\z_+$ such that the actions of $L_i, M_i, Q_{i-\frac{1}{2}}$ for all $i\ge l$ on $W$ are locally finite.
			\item[\rm (2)]   There exists $l\in\z_+$ such that the actions of $L_i, M_i, Q_{i-\frac{1}{2}}$ for all $i\ge l$ on $W$ are locally nilpotent.
			
			\item[\rm (3)]  $W\cong{\rm Ind}_q(V)$ for some $q\in\mathbb N$ and simple $\frak g^{(0, -q)}$-module $V$ such that both conditions $(a)$ and $(b)$ in Theorem \ref{main3}  are satisfied.
  \end{itemize}

	\end{prop}
	\begin{proof} 
		First we prove $(1)\Rightarrow(3)$. Suppose that $W$ is a simple $\mathfrak g$-module and there exists $l\in\mathbb Z_+$ such that the actions of
		$L_i, M_i, Q_{i-\frac{1}{2}},  i\ge l$ are locally finite.  Thus we can find a nonzero element $w\in W$ such that $L_lw = \lambda w$ for some $\lambda\in\mathbb C$.

		For any $j\ge l$, we denote
		\begin{eqnarray*}
			&& V(j)=\sum_{m\in\mathbb N}\mathbb C L_{l}^mM_jw=U(\mathbb C L_{l})M_jw,
		\end{eqnarray*}
		which are all finite-dimensional. By Definition  \ref{LF}, it is clear that $M_{j+(m+1)l}w\in V(j)$ if $M_{j+ml}w\in V(j)$.
		By induction on $m$, we obtain $M_{j+ml}w\in V(j)$ for all $m\in \mathbb N$.
		Hence $\sum_{m\in\mathbb N}\mathbb C M_{j+ml}w$ are
		finite-dimensional for any $j>l$, and then
		\begin{eqnarray*}
			\sum_{i\in\mathbb N}\mathbb C M_{l+i}w=\mathbb C M_{l}w+\sum_{j=l+1}^{2l}\Big(\sum_{m\in\mathbb N}\mathbb C M_{j+ml}w\Big)
		\end{eqnarray*}
		is finite-dimensional.
		
		Now we can take $p\in\mathbb Z_+$ such that
		\begin{equation}
			\sum_{i\in\mathbb N}\mathbb C M_{l+i}w\!=\!\sum_{i=0}^{p}\mathbb C  M_{l+i}w.
		\end{equation}
		Set $$V^\prime:=\sum_{r_0,\ldots,r_p\in \mathbb N}\mathbb C M_{l}^{r_0}\cdots M_{l+p}^{r_p}w.$$  It is clear that $V^\prime$ is finite-dimensional by the condition (1). It follows that we can  choose
		a minimal $n\in\mathbb N$ such that
\begin{equation}\label{lm3.3}			(a_0M_{m}+a_1M_{m+1}+\cdots + a_{n}M_{m+{n}})V^\prime=0
		\end{equation}
		for some $m> l$ and  $a_i\in \mathbb C$.
		Applying $L_m$ to \eqref{lm3.3}, one has		$$(a_0[L_m,M_{m}]+\cdots +a_{n}[L_m,M_{m+{n}}])V^\prime=0$$ since $L_mV^\prime\subset V^\prime$. Then \begin{equation}(a_1M_{2m+1}+\cdots
			+a_{n}nM_{2m+{n}})V^\prime=0.\label{eq311} \end{equation}
   
		Applying suitable $L_k$'s to \eqref{eq311} again, one  has $n=0$, that is, \begin{equation}M_{m}V^\prime=0\label{eq111}\end{equation} for some $m>l$.
		By action of $L_i$ on (\ref{eq111}), we have\begin{equation}M_{m+i}V^\prime=0, \ \forall\, i>0.\end{equation}
		Similarly, we have \begin{equation}L_{n+i}V^\prime=Q_{n+i-\frac12}V^\prime=0\end{equation} for some $n>l$ and any $i>0$.
		
		For any $r, s, t\in \mathbb Z$, we consider the following vector space
		$$N_{r, s, t}=\{v\in W\mid L_iv=M_jv=Q_{k-\frac12}v=0, \quad \forall i>r, j>s, k> t\}.$$
				
		Due to that the action of $M_0+\frac1{24}(n^2-1){\bf c}_2$ is injective on $W$ for any $n\in\mathbb Z^*$, there exists a smallest nonnegative integer, also saying $s$, with $V:=N_{r, s, t}\neq 0$ for some $r\ge s$. 
		Moreover,  we can choose $t=s$ as Theorem \ref{main3}.

		Denote $q=r-s\ge 0$ and $V=N_{s+q, s, s}$.
		For any  $i>s+q, j>s, k>s$, it follows  that
		\begin{eqnarray*}
			&&L_{i}Q_{n-\frac12}v=(\frac {i+1}2-n)Q_{n+i-\frac12}v=0,\quad M_{j}Q_{n-\frac12}v=0,\quad Q_{k-\frac{1}{2}}(Q_{n-\frac12}v)=M'_{k+n-1}v=0,\
		\end{eqnarray*}
		for any $v\in V, n\ge 1$.
		Clearly, $Q_{n-\frac12}v\in V$ for
		all $n\ge 1$. Similarly, we can also obtain $L_kv, M_{k-q}v\in V$
		for all $k\in \mathbb N$. Therefore, $V$ is a $\frak g^{(0, -q)}$-module.

		If $s\ge 1$, by the definition of $V$, the action of $M_{s}$  on $V$ is injective.  Since $W$ is simple
		and generated by $V$, there exists a canonical surjective map
		$$\pi:\mathrm{ Ind}_q(V) \rightarrow W, \quad \pi(1\otimes v)=v,\quad \forall  v\in V.$$
		
  Next, we only need to show that $\pi$ is also injective, that is to say, $\pi$ as the canonical map is bijective.  Let $K=\mathrm{ker}(\pi)$. Obviously, $K\cap V=0$. If $K\neq0$, we can choose a nonzero vector $v\in K$ such that $\mathrm{deg}(v)=(\mi,\mj, \mk)$ is minimal possible.
		Note that $K$ is a $\mathfrak g$-submodule of ${\rm Ind}_q(V)$.
		By the proof of Theorem \ref{main3}, we can obtain another vector $u\in K$  with $\mathrm{deg}(u)\prec(\mi, \mj, \mk)$, which is a contradiction. This forces $K=0$, that is, $W\cong {\rm Ind}_q(V)$. Then $V$ is a simple $\frak g^{(0, -q)}$-module.

		If $s=0$ and $r\ge0 \,(q=r)$.  Similar to the argument above for $s\ge1$ and using the proof of  Theorem \ref{main3} and  the assumption that the action of $M_0+\frac1{24}(n^2-1){\bf c}_2$  on $V$ is injective for any $n\in\mathbb Z^*$, we can deduce that $W\cong {\rm Ind}_q(V)$  and $V$ is a simple $\frak g^{(0, -q)}$-module. 		
		
		Moreover, $(3)\Rightarrow(2)$
		and $(2)\Rightarrow(1)$ are clear. This completes the proof.
	\end{proof}

	\begin{lemm}\label{RL}
		Let $W$ be a simple smooth $\frak g$-module in $\mathcal R_{\frak g}({c_1, c_2})$. Then there exists $N\in\z_+$ such that the actions of $L_i, M_i, Q_{i-\frac{1}{2}}$ for all $i\ge N$ on $W$ are locally nilpotent.
	\end{lemm}
\begin{proof}Let $0\neq v\in W$, there exists $s\in \mathbb{Z}_+$ such that 
$L_iv =M_iv=Q_{i-\frac{1}{2}}v=0$ for all $i\geq s$.
	For $W=U(\mathfrak{g})v$, every element $w$ of $W$ can be uniquely written in the
	following form
	$$
	w=\sum_{\mk\in\mathbb{M}_1,\mi, \mj\in\mathbb{M}}Q^{\mk}M^{\mj}L^{\mi}v_{\mi, \mj, \mk},
	$$ where $v_{\mi, \mj, \mk}\in U(\frak g_+)v$.
	Then, for $i\geq s$, there exists  $N$ sufficiently large such that
	$$
	L_i^Nw=M_i^Nw=Q_{i-\frac{1}{2}}^Nw=0.$$
\end{proof}

	From  Proposition \ref{th2} and Lemma \ref{RL}, we have

	\begin{theo}\label{main4}
		For $c_1, c_2\in\cc$, let $W$ be a simple  $\frak g$-module in $\mathcal{R}_{\mathfrak g}(c_1,c_2)$ with $M_0+\frac1{24}(n^2-1){\bf c}_2$ being injective on $V$ for any $n\in\mathbb Z^*$. Then $W$ is isomorphic to 
		a simple module of the form ${\rm Ind}_q(V)$ in Theorem \ref{main3},  where $q\in \mathbb N$ and $V$ is a simple $\frak g^{(0, -q)}$-module.
	\end{theo}

 \begin{rem}
Let $M$ be a simple smooth ${\rm Vir}$-module with central charge $c_1$. We assume that $M$ has trivial actions of $M_i$, $Q_{i-\frac12}$, and ${\bf c}_2$ for any $i\in\mathbb Z$. As a result, $M$ is also a simple smooth ${\frak g}$-module with central charge $(c,0)$. However, it should be noted that the action of $M_0+\frac1{24}(n^2-1){\bf c}_2$ is not injective.
 \end{rem}

\subsection{Examples of smooth $\mathfrak{g}$-modules}

	For $q,t\in \mathbb{N}$,  set 
	$$\frak g_{q,t}:=\bigoplus_{i>q+t}\mathbb CL_i\oplus \bigoplus_{j>t}(\mathbb CM_j\oplus \mathbb CQ_{j-\frac12}),$$
	we see that  $\frak g_{q,t}$ is an ideal of $\frak g^{(0, -q)}$. Consider the quotient
$
{\frak a}^{(q,t)}=\frak g^{(0, -q)}/ \frak g_{q, t}$, which is a 
finite-dimensional solvable Lie superalgebra.
Theorem  \ref{main4} gives a classification of simple
modules in $\mathcal{R}_{\mathfrak g}(c_1,c_2)$ with $M_0+\frac1{24}(n^2-1){\bf c}_2$ injective action for any $n\in\mathbb Z^*$. 
In order to obtain these simple $\mathfrak{g}$-modules, we have to use simple
modules over ${\frak a}^{(q,t)}$ for all $q,t\in \mathbb{N}$. The classification problem for simple $\mathfrak{a}^{(q,t)}$-modules remains unsolved, as far as we know, except when $(q,t)= (0,0), (1,0), (2,0)$.

\subsubsection{Verma module}

For $(q,t)=(0,0)$, the algebra
$\mathfrak{a}^{(0,0)}$ is commutative and its simple modules are one-dimensional, which leads exactly
to Verma modules.

In details, for $h_i,c_i \in \cc$,  $i=1,2$, 
let $\cc $  be the one-dimensional  $\frak g^{(0, 0)}$-module  defined by
	$$ L_01=h_11, \quad M_01=h_2 1,\quad   {\bf c}_11=c_11,\quad  {\bf c}_21=c_21,\quad {\frak g}_{+}1=0.$$
 It is clear that $\mathbb{C}$ is a simple $\frak g^{(0, 0)}$-module.
It follows from Theorem  \ref{main4} that
$\Ind_0(\mathbb{C})$ is a
simple $\mathfrak{g}$-module in $\mathcal{R}_{\mathfrak g}(c_1,c_2)$  if $h_2+\frac{1}{24}(n^2-1)c_2\neq 0$ for $n\in\mathbb{Z}^*$, which is 
exactly the Verma module  $M_{\mathfrak{g}}(h_1,h_2,c_1,c_2)$ constructed in Theorem \ref{main1}.

\subsubsection{Whittaker module}

	 For any $k\in\mathbb Z_+$, let $\phi_k: \frak g^{(k)}\to\mathbb C$ be  a nontrivial Lie superalgebra homomorphism with $\phi_k({\bf c}_1)=c_1, \phi_k({\bf c}_2)=c_2$. It follows that
$$\phi_k(L_i) = \phi_k(M_i)=0, \quad\forall i \ge 2k+1,\quad\text{and}\quad \phi_k(Q_{j+\frac12})=0,\quad \forall j\ge k.$$
Let $\cc w$ be the one-dimensional $\frak g^{(k)}$-module with $xw=\phi_k(x)w$ for all $x\in\frak g^{(k)}$. The universal Whittaker module $W (\phi_k)$ can be defined as
$$
W (\phi_k)= U(\mathfrak{g})\otimes_{U(\mathfrak{g}^{(k)})}\mathbb C w.
$$

Let $V={\rm Ind}_{\frak g^{(k)}}^{\frak g^{(0, 0)}}(\mathbb Cw)$. Applying similar arguments as in the proof of Theorem \ref{main3} we can prove that   $V$ is a simple $\frak g^{(0, 0)}$-module if and only if $\phi_{k}(M_{2k-1})\neq 0$ or
 $\phi_{k}(M_{2k})\neq0$. It follows that  
$$W (\phi_k)={\rm Ind}_{\frak g^{(0, 0)}}^{\frak g}(V)=\Ind_0(V).$$

As an application of Theorem \ref{main3}  (here $(q, t)=(0, 2k-1)\  \text{or}\  (0,2k)$),  we can obtain the following result.
\begin{coro}[\cite{DGL1}]\label{whittaker} For $k\in \mathbb{Z}_+$, the universal  Whittaker module $W (\phi_k)$ is a simple  $\mathfrak{g}$-module in $\mathcal{R}_{\mathfrak g}(c_1,c_2)$ if and only if $\phi_k(M_{2k})\neq 0$ or $\phi_k(M_{2k-1})\neq0$.
	\end{coro}

			\section{Free field realizations of smooth  $\mathfrak{g}$-modules}
	
	In this section, we construct a  free field realization of the  N=1 BMS superalgebra  via the Heisenberg-Clifford algebra $\mathfrak{hc}$.
	Using this free field realization, we construct some smooth modules over the N=1 BMS superalgebra from smooth modules over the Heisenberg-Clifford superalgebra.

		\subsection{N=1 BMS vertex superalgebra}

	Recall that a {vertex  superalgebra} (cf. \cite{KW,Li}) is a quadruple $(V, {\mathds1}, D, Y )$, where $V $ is a
	$\frac{1}{2}\z$-graded super vector space
	\begin{equation*}
		V=\bigoplus_{n\in{ \frac{1}{2}\z}}V_n= V_{\bar{0}}\oplus V_{\bar{1}}
	\end{equation*}
	with  $V_{\bar{0}}=\sum_{n\in\z}V_n$ and
	$V_{\bar{1}}=\sum_{n\in\z+\frac{1}{2}}V_n$,  ${\mathds1}$ is a specified vector called
	the vacuum of $V$, $D$  is an endomorphism of $V$, and $Y$ is a linear map
	\begin{align*}
		& V \to (\mbox{End}\,V)[[z,z^{-1}]] ,\\
		& v\mapsto Y(v,z)=\sum_{n\in{\z}}v_nz^{-n-1}\ \ \ \  (v_n\in
		(\End\, V)_{\tilde v}),\nonumber
	\end{align*}
	satisfying the following conditions for $u, v \in V,$ and $m,n\in\z:$
	\begin{itemize}
		\item[(V1)] $u_nv=0$\ \ \ \ \ {\rm for}\ \  $n$\ \ {\rm sufficiently\ large};
		\item[(V2)]  $Y({\mathds1},z)={\rm Id}_{V}$;
		\item[(V3)]  $Y(v,z){\mathds1}\in V[[z]]$\ \ \ {\rm and}\ \ \ $\lim_{z\to
			0}Y(v,z){\mathds1}=v$;
		\item[(V4)] $\frac{d}{dz}Y(v,z)=Y(Dv,z)=[D,Y(v,z)]$;
		\item[(V5)]  For $\z_2$-homogeneous $u,v\in V$,  the following {Jacobi identity} holds:
		\begin{equation*}\label{2.8}
			\begin{array}{c}
		\displaystyle{z^{-1}_0\delta\left(\frac{z_1-z_2}{z_0}\right)
					Y(u,z_1)Y(v,z_2)-(-1)^{{|u|}{|v|}}z^{-1}_0\delta\left(\frac{z_2-z_1}{-z_0}\right)
					Y(v,z_2)Y(u,z_1)}\\
				\displaystyle{=z_2^{-1}\delta
					\left(\frac{z_1-z_0}{z_2}\right)
					Y(Y(u,z_0)v,z_2)},
			\end{array}
		\end{equation*}
		where
		$\delta(z)=\sum_{n\in {\z}}z^n$ and  $(z_i-z_j)^n$ is
		expanded as a formal power series in $z_j$.
	\end{itemize}
	
	A { homomorphism} from a vertex  superalgebra $V_1$ to another vertex
	superalgebra $V_2$ is a linear map $f:V_1 \to V_2$ such that
	$$f({\mathds1}_1) = {\mathds1}_2,\quad  f( Y_1(u,z)v) =Y_2(f(u),z)f(v),\quad\forall u,v\in V_1.$$

	For $u,v\in V$, we define the normal order $\NO\ \NO$ of vertex operators as follows:
	$$
	\NO Y(u,z_1)Y(v,z_2)\NO=Y^+(u,z_1)Y(v,z_2)+(-1)^{|u||v| }Y(v,z_2)Y^-(u,z_1),
	$$
	where
	$$
	Y(u,z)=Y^-(u,z)+Y^+(u,z)=\sum_{n\geq0}u_nz^{-n-1}+\sum_{n<0}u_nz^{-n-1}.
	$$

	Let $V$ be a vertex superalgebra. A $V$-module is a triple
	$(M,d,Y_M)$,  where $M$ is a $\z_2$-graded vector space, $d$ is an endomorphism of $M$, and $Y_M$
	is a linear map
	$$\begin{array}{l}
		V\to (\End\,M)[[z, z^{-1}]],\\
		v\mapsto\displaystyle{ Y_M(v,z)=\sum_{n\in\z}v_nz^{-n-1}\ \ \ (v_n\in
			\End\,M)},
	\end{array}$$
	which satisfies that for all  $u, v\in V,$ $w\in M,$
	\begin{itemize}
		\item[(M1)] $u_nw=0$ \    {\rm for}\ \  $n$\ \ {\rm sufficiently\ large};
		\item[(M2)]$Y_M({\mathds1},z)={\rm Id}_{M}$;
		\item[(M3)]$\frac{d}{dz}Y_M(v,z)=Y_M(Dv,z)=[d,Y_M(v,z)]$;
		\item[(M4)] For $\z_2$-homogeneous $u,v$, the following Jacobi identity holds:
		\begin{equation*}\label{2.14}
			\begin{array}{c}
				\displaystyle{z^{-1}_0\delta\left(\frac{z_1-z_2}{z_0}\right)
					Y_M(u,z_1)Y_M(v,z_2)-(-1)^{|{u}||{v}|}z^{-1}_0\delta\left(\frac{z_2-z_1}{-z_0}\right)
					Y_M(v,z_2)Y_M(u,z_1)}\\
				\displaystyle{=z_2^{-1}\delta\left(\frac{z_1-z_0}{z_2}\right)
					Y_M(Y(u,z_0)v,z_2)}.
			\end{array}
		\end{equation*}
	\end{itemize}

	We shall see that there is
	a canonical vertex superalgebra structure of  on the vacuum module $V_{\frak g}(c_1,c_2)$.

\begin{lemm}	Set
	$$
	L(z)=\sum_{n\in\z}L_nz^{-n-2},\
	M(z)=\sum_{n\in\z}M_nz^{-n-2},\
	Q(z)=\sum_{r\in\z+\frac{1}{2}}Q_rz^{-r-\frac{3}{2}} \in {\rm End}(V_{\frak g}(c_1,c_2))[[z,z^{-1}]].
	$$ Then the defining relations for the N=1 BMS superalgebra $\mathfrak{g}$ are equivalent to the following operator product expansions (OPEs).
 \begin{eqnarray*}
			L(z_1)L(z_2)&\sim&\frac{\frac{c_1}{2}}{(z_1-z_2)^4}+\frac{2L(z_2)}{(z_1-z_2)^2}+\frac{L'(z_2)}{z_1-z_2},\\
			{L(z_1)M(z_2)}&\sim& \frac{\frac{c_2}{2}}{(z_1-z_2)^4}+\frac{2M(z_2)}{(z_1-z_2)^2}+\frac{M'(z_2)}{z_1-z_2},\\
			{L(z_1)Q(z_2)}&\sim & \frac{\frac{3}{2}Q(z_2)}{(z_1-z_2)^2}+\frac{Q'(z_2)}{z_1-z_2},\\
			{Q(z_1)Q(z_2)}&\sim &\frac{\frac{2c_2}{3}}{(z_1-z_2)^3}+\frac{2M(z_2)}{z_1-z_2}.
		\end{eqnarray*}
Furthermore, $\{L(z), M(z), Q(z)\}$ is a  set of mutually local homogeneous vertex operators on every smooth module over the N=1 BMS superalgebra $\mathfrak{g}$.
\end{lemm}

	\begin{proof}
		The following operator product expansions (OPEs) can be verified directly.
It follows that
		$$
		(z_1-z_2)^4[L(z_1),L(z_2)]=(z_1-z_2)^4[L(z_1),M(z_2)]=0,
		$$
		$$
		(z_1-z_2)^2[L(z_1),Q(z_2)]=(z_1-z_2)^3[Q(z_1),Q(z_2)]=0.
		$$
	\end{proof}

	By  the local system theory  for vertex superalgebras developed in  \cite{Li}, we get the following result.

	\begin{prop} \label{RM1}For $c_1,c_2\in\cc$ ,
		$V_{\frak g}(c_1,c_2)$  has a  vertex
		superalgebra structure,  which is uniquely
		determined by the condition that ${\mathds{1}}$ is the vacuum vector, $Dv=L_{-1}v$,  and
		$$
		Y(L_{-2}\mathds{1},z)=L(z), \quad Y(M_{-2}{\mathds{1}},z)=M(z),\quad Y(Q_{-\frac{3}{2}}{\mathds{1}},z)=Q(z).
		$$
		Moreover, 
there is a one-to-one correspondence between smooth $\frak g$-modules   of central charge  $(c_1, c_2)$
 and  $V_{\frak g}(c_1,c_2)$-modules.  
	\end{prop}

\subsection{Free field realizations}

\begin{defi}
		The { Heisenberg-Clifford algebra} $\mathfrak{hc}=\mathfrak{hc}_{\bar 0}\oplus \mathfrak{hc}_{\bar 1}$ is the Lie superalgebra generated by $a_n,b_n, c_r,{\bf k}, n\in\z, r\in\z+\frac{1}{2}$ with
	\begin{eqnarray*}\label{HC}
		&&[a_m,b_n]=m\delta_{m+n,0}{\bf k},\quad \{c_r,c_s\}=\delta_{r+s,0}{\bf k},\label{HC}\\
		&&[a_m,a_n]=[b_m,b_n]=[a_m,c_r]=[b_m,c_r]=0,\quad [{\bf k}, \mathfrak{hc}]=0,
	\end{eqnarray*}
	where $\mathfrak{hc}_{\bar 0}=\bigoplus_{n\in\z}\cc a_n\oplus\bigoplus_{n\in\z}\cc b_n\oplus\cc {\bf k}$ and $\mathfrak{hc}_{\bar 1}=\bigoplus_{r\in\z+\frac{1}{2}}\cc c_r$.
\end{defi}
 It is clear that $\mathfrak{hc}$ is  generated by free bosons $\{a_n,b_n, {\bf k}\mid n\in\z\}$  and  neutral fermions $\{c_n,{\bf k}\mid n\in \z+\frac{1}{2}\}$.

Define
	$$
	{\rm deg}\, a_n={\rm deg}\, b_n=n,\quad  {\rm deg}\, c_r=r, \quad\forall n\in\z, r\in\z+\frac{1}{2}.
	$$
	Then $\mathfrak{hc}$ is a $\frac{1}{2}\z$-graded Lie superalgebra. Set
	$$
	\mathfrak{hc}_{\pm}=\bigoplus_{n\in\z_+}\cc a_{\pm n}\oplus\bigoplus_{n\in\z_+}\cc b_{\pm n}\oplus  \bigoplus_{r\in\mathbb{N}+\frac{1}2}\cc c_{\pm r},\quad \mathfrak{hc}_{0}=\cc a_0\oplus\cc b_0\oplus\cc {\bf k}.
	$$
	Then
	$\mathfrak{hc}=\mathfrak{hc}_{+}\oplus \mathfrak{hc}_{0}\oplus \mathfrak{hc}_{-}.$ Note that $\mathfrak{hc}_0$ is the center of the superalgebra $\mathfrak{hc}$.

  For $a,b,\ell\in\cc$, let $\cc $  be the one-dimensional  ${\mathfrak{hc}}_{0}$-module  defined by
	$$ a_01=a1, \quad b_01=b 1,\quad   {\bf k}1=\ell 1.$$
	Let ${\mathfrak{hc}}_{+}$ act trivially on $\cc$, making $\cc$  be a $({\mathfrak{hc}}_{+}\oplus{\mathfrak{hc}}_{0})$-module.
	The {\bf  Verma module} for ${\mathfrak{hc}}$ is defined by
	$$ M_{\mathfrak{hc}}(\ell,a,b)=U({\mathfrak{hc}})\otimes_{U({\mathfrak{hc}}_{+}\oplus{\mathfrak{hc}}_{0})}\cc .$$

	An ${\mathfrak{hc}}$-module $M$ is
	said to be {\bf smooth} if for any $u\in M$, $a_iu=b_j u=c_ku=0$ for $i,j,k$ sufficiently large. For $\ell\in \mathbb{C}$, an ${\mathfrak{hc}}$-module $M$ is
	said to be {\bf level} $\ell$ if for any $u\in M$, ${\bf k} u=\ell u$. It is clear that the Verma module $M_{\mathfrak{hc}}(\ell,a,b)$ is a smooth $\mathfrak{g}$-module of level $\ell$.

 Suppose that $\phi: \mathfrak{h} \mathfrak{c}_{+} \oplus \mathfrak{h} \mathfrak{c}_0 \rightarrow \mathbb{C}$ is a homomorphism of Lie superalgebras. It follows that $\phi\left(c_r\right)=0$ for all $r>0$. Then $\mathbb{C} w_\phi$ becomes a one-dimensional $\left(\mathfrak{h} \mathfrak{c}_{+} \oplus \mathfrak{h} \mathfrak{c}_0\right.$)-module defined by $x w_\phi=\phi(x) w_\phi$ for all $x \in \mathfrak{h} \mathfrak{c}_{+} \oplus \mathfrak{h} \mathfrak{c}_0$. The induced $\mathfrak{hc}$-module $W_{\mathfrak{hc}}(\phi)=\mathrm{Ind}_{\mathfrak{hc}_{+} \oplus \mathfrak{h} \mathfrak{c}_0}^{\mathfrak{hc}} \mathbb{C} w_\phi$ is called a {\bf Whittaker module} associated to $\phi$. Note that $W_{\mathfrak{hc}}(\phi)$ is not necessarily a smooth $\mathfrak{hc}$-module.

	Following \cite{LPXZ}, we have

 \begin{lemm}\label{chc}
 The Whittaker module $W_{\mathfrak{hc}}(\phi)$ is simple if and only if $\phi(\mathbf{k}) \neq 0$.
 \end{lemm}
 

	It is clear that  there is
	a canonical vertex superalgebra structure of  on $M_{\mathfrak{hc}}(1,0,0)$.
	\begin{prop}[\cite{Li}]\label{RM2}
		$M_{\mathfrak{hc}}(1,0,0)$  has a  vertex superalgebra structure with ${\mathds{1}}=1\otimes 1$ and
		$$
		Y(a_{-1}{\mathds{1}},z)=a(z),\quad Y(b_{-1}{\mathds{1}},z)=b(z),\quad Y(c_{-\frac{1}{2}}{\mathds{1}},z)=c(z),
		$$
		where
		$$
		a(z)=\sum_{n\in\z}a_nz^{-n-1},\
		b(z)=\sum_{n\in\z}b_nz^{-n-1},\
		c(z)=\sum_{r\in\z+\frac{1}{2}}c_rz^{-r-\frac{1}{2}} \in {\rm End} (M_{\mathfrak{hc}}(1,0,0)) [[z,z^{-1}]].
		$$
		Moreover,  there is a one-to-one correspondence between smooth $\frak hc$-modules   of central charge  $1$  and $M_{\mathfrak{hc}}(1,0,0)$-modules.

	\end{prop}

	For $u,v\in M_{\mathfrak{hc}}(1,0,0)$,   a {\bf contraction} of the vertex operators $Y(u,z_1)$ and $Y(v,z_2)$   is defined by
		$$
		\underbrace{Y(u,z_1)Y(v,z_2)}=Y(u,z_1)Y(v,z_2)-\NO Y(u,z_1)Y(v,z_2)\NO.
		$$
		It follows that
		$$
		\underbrace{a(z_1)b(z_2)}=\underbrace{b(z_1)a(z_2)}= \frac{1}{(z_1-z_2)^2},\quad
		\underbrace{c(z_1)c(z_2)}=\frac{1}{z_1-z_2}.
		$$

	For simplicity, let $a=a_{-1}{\mathds{1}}, b=b_{-1}{\mathds{1}},c=c_{-\frac{1}{2}}{\mathds{1}}\in M_{\mathfrak{bc}}(1,0,0)$.
	
	\begin{prop}\label{HTH}
		For $\rho\in\cc $, there is a homomorphism of vertex   superalgebras
		$$
		\Phi_{\rho}: V_{\frak g}\left(\frac{5}{2},-12\rho^2\right)\to M_{\mathfrak{hc}}(1,0,0)
		$$
		uniquely determined by
  \begin{align*}
      L_{-2}{\mathds{1}}&\mapsto a_{-1}b+\rho a_{-2}{\mathds{1}}+\frac{1}{2}c_{-\frac{3}{2}}c,\\
       M_{-2}{\mathds{1}}&\mapsto \frac{1}{2}b_{-1}b+\rho b_{-2}{\mathds{1}},\\
        Q_{-\frac{3}{2}}{\mathds{1}}&\mapsto b_{-1}c+\rho c_{-\frac{3}{2}}{\mathds{1}}.
  \end{align*}
  or  equivalently, 
		\begin{align*}
			L(z)&\mapsto \mathcal{L}(z)=\NO a(z)b(z)\NO+ \rho a'(z)+ \frac{1}{2}\NO c'(z)c(z)\NO,\\
			M(z)&\mapsto  \mathcal{M}(z)=\frac{1}{2} \NO b(z)b(z) \NO+ \rho b'(z),\\
			Q(z)&\mapsto \mathcal{Q}(z)=\NO b(z)c(z)\NO+ 2\rho c'(z).
		\end{align*}
	\end{prop}
	\begin{proof} 
		By the Wick theorem  in \cite{Kac1}, we have
		\begin{eqnarray*}
			&&\underbrace{\NO a(z_1)b(z_1)\NO \NO a(z_2)b(z_2)\NO }\\
			&=& \underbrace{a(z_1)b(z_2)}\NO b(z_1)a(z_2)\NO +\underbrace{b(z_1)a(z_2)}\NO a(z_1)b(z_2)\NO + \underbrace{a(z_1)b(z_2)}\underbrace{b(z_1)a(z_2)}\\
			&=&\frac{\NO b(z_1)a(z_2)\NO+\NO a(z_1)b(z_2)\NO}{(z_1-z_2)^2}+\frac{1}{{(z_1-z_2)^4}}\\
			&=&\frac{2\NO a(z_2)b(z_2)\NO}{(z_1-z_2)^2}+\frac{\NO a'(z_2)b(z_2)\NO+\NO b'(z_2)a(z_2)\NO}{z_1-z_2}+\frac{1}{{(z_1-z_2)^4}};
		\end{eqnarray*}
		
		\begin{eqnarray*}
			&&\underbrace{\NO a(z_1)b(z_1)\NO a'(z_2)}+\underbrace{a'(z_1) \NO a(z_2)b(z_2)\NO} \\
			&= &\underbrace{a(z_1)a'(z_2)}b(z_1)+\underbrace{b(z_1)a'(z_2)}a(z_1)+ \underbrace{a'(z_1)a(z_2)}b(z_2)+\underbrace{a'(z_1)b(z_2)}a(z_2)\\
			&=&\frac{2}{{(z_1-z_2)^3}}a(z_2)+\frac{2}{{(z_1-z_2)^2}}a'(z_2)+\frac{1}{{z_1-z_2}}a''(z_2)-\frac{2}{{(z_1-z_2)^3}}a(z_2)\\
			&=&\frac{2}{{(z_1-z_2)^2}}a'(z_2)+\frac{1}{{z_1-z_2}}a''(z_2);
		\end{eqnarray*}

		\begin{eqnarray*}
			&&\underbrace{\NO c'(z_1)c(z_1)\NO  \NO c'(z_2)c(z_2)\NO}\\
			&=& -\underbrace{c'(z_1)c'(z_2)}\NO c(z_1)c(z_2)\NO+\underbrace{c'(z_1)c(z_2)}\NO c(z_1)c'(z_2)\NO\\
			&&+\underbrace{c(z_1)c'(z_2)}\NO c'(z_1)c(z_2)\NO-\underbrace{c(z_1)c(z_2)}\NO c'(z_1)c'(z_2)\NO\\
			&&-\underbrace{c'(z_1)c'(z_2)}\underbrace{c(z_1)c(z_2)}+\underbrace{c'(z_1)c(z_2)}\underbrace{c(z_1)c'(z_2)}\\
			&=&\frac{2}{(z_1-z_2)^3}\NO c(z_1)c(z_2)\NO+\frac{-1}{(z_1-z_2)^2}\NO c(z_1)c'(z_2)\NO\\
			&&+\frac{1}{(z_1-z_2)^2}\NO c'(z_1)c(z_2)\NO-\frac{1}{z_1-z_2}\NO c'(z_1)c'(z_2)\NO+\frac{2}{{(z_1-z_2)^4}}-\frac{1}{{(z_1-z_2)^4}}\\
			&=&\frac{2}{(z_1-z_2)^3}\NO c(z_2)c(z_2)\NO+\frac{2}{(z_1-z_2)^2}\NO c'(z_2)c(z_2)\NO+\frac{1}{z_1-z_2}\NO c''(z_2)c(z_2)\NO\\
			&&-\frac{1}{(z_1-z_2)^2}\NO c(z_2)c'(z_2)\NO-\frac{1}{z_1-z_2}\NO c'(z_2)c'(z_2)\NO\\
			&&+\frac{1}{(z_1-z_2)^2}\NO c'(z_2)c(z_2)\NO+\frac{1}{z_1-z_2}\NO c''(z_2)c(z_2)\NO\\
			&&-\frac{1}{z_1-z_2}\NO c'(z_2)c'(z_2)\NO+\frac{1}{{(z_1-z_2)^4}}\\
			&=&\frac{4\NO c'(z_2)c(z_2)\NO}{(z_1-z_2)^2}+\frac{2\NO c''(z_2)c(z_2)\NO}{z_1-z_2}+\frac{1}{{(z_1-z_2)^4}};
		\end{eqnarray*}

		\begin{eqnarray*}
			&&\underbrace{\mathcal{L}(z_1)\mathcal{L}(z_2)}\\
			&=&\frac{2\NO a(z_2)b(z_2)\NO}{(z_1-z_2)^2}+\frac{\NO a'(z_2)b(z_2)\NO+\NO b'(z_2)a(z_2)\NO}{z_1-z_2}+\frac{1}{{(z_1-z_2)^4}}\\
			&&+\frac{2\rho a'(z_2)}{{(z_1-z_2)^2}}+\frac{\rho a''(z_2)}{{z_1-z_2}}+\frac{2\frac{1}{2}\NO c'(z_2)c(z_2)\NO}{(z_1-z_2)^2}+\frac{\frac{1}{2}\NO c''(z_2)c(z_2)\NO}{z_1-z_2}+\frac{\frac{1}{4}}{{(z_1-z_2)^4}}\\
			&=&\frac{2\left(\NO a(z_2)b(z_2)\NO+2\rho a'(z_2)+\frac{1}{2}\NO c'(z_2)c(z_2)\NO+\frac{1}{2}\NO c'(z_2)c(z_2)\NO\right)}{(z_1-z_2)^2}\\
			&&+\frac{\NO a'(z_2)b(z_2)\NO+\NO b'(z_2)a(z_2)\NO+\rho a''(z_2)+\frac{1}{2}\NO c''(z_2)c(z_2)\NO}{z_1-z_2}+\frac{\frac{5}{4}}{{(z_1-z_2)^4}}\\
			&=&\frac{\frac{5}{4}}{(z_1-z_2)^4}+\frac{2\mathcal{L}(z_2)}{(z_1-z_2)^2}+\frac{\mathcal{L}'(z_2)}{z_1-z_2};
		\end{eqnarray*}

		\begin{eqnarray*}
			&&\underbrace{\mathcal{L}(z_1)\mathcal{M}(z_2)}
			\\
			&=&\underbrace{a(z_1)b(z_2)}   b(z_1)b(z_2)+\rho \underbrace{a(z_1) b'(z_2)} b(z_1) +\rho \underbrace{a'(z_1)b(z_2)}+\rho^2 a'(z_1)b'(z_2)\\
			&=&\frac{b(z_2)b(z_2)}{(z_1-z_2)^2}+\frac{b'(z_2)b(z_2)}{z_1-z_2} +\frac{2\rho}{(z_1-z_2)^3} b(z_1)\\
			&&+\frac{2\rho b'(z_2)}{(z_1-z_2)^2}+\frac{\rho b''(z_1)}{z_1-z_2}+\frac{-2\rho}{(z_1-z_2)^3}+\frac{-6\rho^2}{(z_1-z_2)^4}\\
			&=&\frac{ b(z_2)b(z_2)+2\rho b'(z_2)}{(z_1-z_2)^2}+\frac{b'(z_2)b(z_2)+\rho b''(z_1)}{z_1-z_2} +\frac{-6\rho^2}{(z_1-z_2)^4}\\
			&=&\frac{-6\rho^2}{(z_1-z_2)^4}+\frac{2\mathcal{M}(z_2)}{(z_1-z_2)^2}+\frac{\mathcal{M}'(z_2)}{z_1-z_2};
		\end{eqnarray*}

\begin{eqnarray*}
			&&\underbrace{\mathcal{L}(z_1)\mathcal{Q}(z_2)}\\
			&=&\frac{b(z_2)c(z_2)}{(z_1-z_2)^2}+\frac{b'(z_2)c(z_2)}{z_1-z_2}+\frac{-2\rho c(z_2)}{(z_1-z_3)^3}+\frac{\frac{1}{2}c(z_2)b(z_2)}{(z_1-z_2)^2}+\frac{\frac{1}{2}c'(z_2)b(z_2)}{z_1-z_2}+\frac{\frac{1}{2}c'(z_2)b(z_2)}{z_1-z_2}\\
			&&+\frac{2\rho c(z_2)}{(z_1-z_2)^3}+\frac{2\rho c'(z_2)}{(z_1-z_2)^2}+\frac{\rho c''(z_2)}{z_1-z_2}+\frac{\rho c'(z_2)}{(z_1-z_2)^2}+\frac{\rho c''(z_2)}{z_1-z_2}\\
			&=&\frac{\frac{3}{2}b(z_2)c(z_2)+3\rho c'(z_2)}{(z_1-z_2)^2}+\frac{b'(z_2)c(z_2)+c'(z_2)b(z_2)+2\rho c''(z_2)}{z_1-z_2}\\
			&=& \frac{\frac{3}{2}\mathcal{Q}(z_2)}{(z_1-z_2)^2}+\frac{\mathcal{Q}'(z_2)}{z_1-z_2};
		\end{eqnarray*}
  and	
		\begin{eqnarray*}
			&& \underbrace{\mathcal{Q}(z_1)\mathcal{Q}(z_2)}\\
			&=&\frac{b(z_1)b(z_2)}{z_1-z_2}+\frac{2\rho b(z_2)}{(z_1-z_2)^2}+\frac{2\rho b'(z_2)}{z_1-z_2}+\frac{-2\rho b(z_2)}{(z_1-z_2)^2}+\frac{-8\rho^2}{(z_1-z_2)^3}\\
			&=&\frac{2\rho b'(z_2)+b(z_2)b(z_2)}{z_1-z_2}+\frac{-8\rho^2}{(z_1-z_2)^3}\\
			&=&\frac{-8\rho^2}{(z_1-z_2)^3}+\frac{2\mathcal{M}(z_2)}{z_1-z_2}.
		\end{eqnarray*}

	\end{proof}

		We are now in a position to state the main result of this section.

	\begin{theo}\label{main2} For $\rho\in \cc$, every smooth $\mathfrak{hc}$-module $W$ of level $1$ becomes a smooth $\frak g$-module of central charge $\left(\frac{5}{2},-12\rho^2\right)$  with the following actions:  
		\begin{align*}
			L_n& \mapsto  \sum_{m\in\z} \NO a_{m}b_{n-m}\NO-(n+1)\rho a_n- \frac{1}{2}\sum_{s\in\z+\frac{1}{2}}(s+\frac{1}{2})\NO c_{s}c_{n-s}\NO,\\
			M_n&\mapsto\frac{1}{2}\sum_{m\in\z} b_{m}b_{n-m}-(n+1)\rho b_n,\\
			Q_r&\mapsto  \sum_{s\in\z+\frac{1}{2}}(r+\frac{1}{2}) b_{r-s}c_{s} -2(r+\frac12)\rho c_{r},\\
			{\bf c}_1 &\mapsto \frac52,\\
			{\bf c}_2 &\mapsto -12\rho^2,		\end{align*}
		 for any $ n\in\z,r\in\z+\frac12, \rho\in \cc$. 
	\end{theo}
	\begin{proof}
		By Proposition \ref{RM2},  any smooth $\mathfrak{hc}$-module $W$ of level $1$  is  an $M_{\mathfrak{hc}}(1,0,0)$-module.    From  Proposition \ref{HTH},      $W$ becomes  a  $V_{\frak g}\left(\frac{5}{2},-12\rho^2\right)$-module.  By Proposition \ref{RM1}, $W$  is a smooth $\frak g$-module of central charge $\left(\frac{5}{2},-12\rho^2\right)$.
	\end{proof}

By Theorem \ref{main2},
the Verma module $M_{\mathfrak{hc}}(1,a,b)$  is a  $\frak g$-module of central charge $\left(\frac{5}{2},-12\rho^2\right)$. 
We refer to this module as a {\bf Fock module} over $\mathfrak{g}$, denoted by $\mathcal{F}_{\mathfrak{g}}(a,b,\rho)$.

	\begin{coro} For $a,b,\rho\in\cc$,
		the Fock $\frak g$-module  $\mathcal{F}_{\mathfrak{g}}(a,b,\rho)$ is simple if and only if
		$b+(n-1)\rho\ne 0$ for any $n\in\mathbb Z^*$.
	\end{coro}
	\begin{proof}By Theorem \ref{main2},  we have
		$$
		L_0{\bf 1}=(ab-\rho a){\bf 1},\quad  M_0{\bf 1}=\left(\frac{1}{2}b^2-\rho b\right){\bf 1}, \quad {\bf c}_1{\bf 1}=\frac52{\bf 1},\quad  {\bf c}_2{\bf 1}=-12\rho^2{\bf 1}.
		$$
  It follow that  $\mathcal{F}_{\mathfrak{g}}(a,b,\rho)$  is a  highest weight $\frak g$-module of central charge $\left(\frac{5}{2},-12\rho^2\right)$.
		From the universal property of the Verma module $M_{\frak g}(ab-\rho a, \frac{1}{2}b^2-\rho b, \frac{5}{2},-12\rho^2)$  and
		$$
		\dim M_{\frak g}\left(ab-\rho a, \frac{1}{2}b^2-\rho b, \frac{5}{2},-12\rho^2\right)_n=  \dim M_{\mathfrak{hc}}(1,a,b)_n=p(n),
		$$
		we have  an isomorphism of $\frak g$-modules :
		$$M_{\frak g}(ab-\rho a, \frac{1}{2}b^2-\rho b, \frac{5}{2},-12\rho^2)\cong M_{\mathfrak{hc}}(1,a,b).$$
		It follows from Theorem \ref{main1} that $\frak g$-module $M_{\mathfrak{hc}}(1,a,b)$  is simple if and only if
		$$
		\frac{1}{2}b^2-\rho b-\frac{1}{2}(i^2-1)\rho^2\neq 0,\quad \forall i\in \z_+,
		$$
		 if and only if $$(b+(i-1)\rho )(b-(i+1)\rho)\ne0,$$
		  if and only if
		$b+(n-1)\rho\ne 0$ for any $n\in\mathbb Z^*$.
		\end{proof}

Let $W_{\mathfrak{hc}}(\phi)$ be the Whittaker $\mathfrak{hc}$-module associated to $\phi$ with $\phi({\bf k})=1$.    Suppose that  
$$
a_{i+1}w_{\phi}=b_{i+1}w_{\phi}=0,\quad \forall i\in\mathbb Z_+.
$$ 
Note $c_{n-\frac12}w=0$ for any $n\in\mathbb Z_+$.  By Lemma \ref{chc}, $W_{\mathfrak{hc}}(\phi)$  is a simple smooth $\frak{hc}$-module of level $1$.  For $\rho\in \cc$, it follows Theorem \ref{main2} that $W_{\mathfrak{hc}}(\phi)$ becomes a $\frak g$-module of central charge $\left(\frac{5}{2},-12\rho^2\right)$.
We refer to this module as a {\bf Fock-Whittaker module} over $\mathfrak{g}$, denoted by $\mathcal{F}_{\mathfrak{g}}(\phi,\rho)$.

\begin{coro} The Fock-Whittaker $\mathfrak{g}$-module  $\mathcal{F}_{\mathfrak{g}}(\phi,\rho)$   is  simple if and only if $\phi(b_1)\ne 0$.

\end{coro}
\begin{proof}
    By Theorem \ref{main2}, we have
    \begin{align*}
L_1w_{\phi}&=(\phi(a_0)\phi(b_1)+\phi(a_1)\phi(b_0)-2\rho \phi(a_1))w_{\phi},\ 
L_2w_{\phi}=\phi(a_1)\phi(b_1)w_{\phi},\  L_iw_{\phi}=0,\\
M_1w_{\phi}&=(\phi(b_0)-2\rho)\phi(b_1)w_{\phi},\  M_2w_{\phi}=\frac12\phi(b_1)^2w_{\phi},\  M_iw_{\phi}=0,\\ 
Q_{\frac12}w_{\phi}&=\phi(b_1)c_{-\frac12}w_{\phi},\  Q_{r}w_{\phi}=0,\  {\bf c}_1w_{\phi}=\frac52w_{\phi},\   {\bf c}_2w_{\phi}=-12\rho^2w_{\phi}
\end{align*}
for any $ i\ge 3, r\geq \frac{3}{2}$. It is clear that $\mathcal{F}_{\mathfrak{g}}(\phi,\rho)$ is a Whittaker $\mathfrak{g}$-module of central charge $\left(\frac{5}{2},-12\rho^2\right)$.  From Corollary \ref{whittaker}, $\mathcal{F}_{\mathfrak{g}}(\phi,\rho)$   is  simple if and only if $\phi(b_1)\ne 0$.
\end{proof}

	\section*{Acknowledgements}
The authors would like to thank Prof. Haisheng Li for his useful and helpful comments on our manuscript.
This work was supported by the National Natural Science Foundation of China (11971315, 12071405,12171155) and NSERC (311907-2020).

\end{document}